\documentclass{article}
\usepackage[utf8]{inputenc}
\usepackage[T2A]{fontenc}
\usepackage[a4paper]{geometry}
\usepackage{amsmath,amssymb,amsthm}

\usepackage{ifthen}

\usepackage{wrapfig}

\usepackage[normalem]{ulem}

\usepackage{hyperref}
\hypersetup{colorlinks=true,linkcolor=blue,citecolor=teal,
filecolor=magenta,urlcolor=cyan}

\usepackage{tikz}
\usetikzlibrary{calc,decorations.markings}

\usepackage{graphicx}
\usetikzlibrary{patterns}

\newtheorem{theorem}{Theorem}[section]
\newtheorem{remark}{Remark}[section]
\newtheorem{corollary}[theorem]{Corollary}
\newtheorem{proposition}[theorem]{Proposition}
\newtheorem{lemma}{Lemma}[section]
\newtheorem{example}{Example}[section]

\newtheorem{definition}{Definition}[section]

\renewcommand{\d}{\partial}

\newcommand{\BC}{\mathbb{C}}

\newcommand{\BS}{\mathbb{S}}

\newcommand{\BZ}{\mathbb{Z}}

\newcommand{\cA}{\mathcal{A}}

\newcommand{\cC}{\mathcal{C}}

\newcommand{\cH}{\mathcal{H}}

\newcommand{\cS}{\mathcal{S}}

\newcommand{\fg}{\mathfrak{g}}

\newcommand{\fh}{\mathfrak{h}}

\newcommand{\fn}{\mathfrak{n}}

\newcommand{\SL}{\mathfrak{sl}}

\newcommand{\so}{\mathfrak{so}}

\newcommand{\fgl}{\mathfrak{gl}}
\newcommand{\GL}{\mathfrak{gl}}

\def\gl{\mathfrak{gl}}
\def\sl{\mathfrak{sl}}

\title{Generalized chord diagrams and weight systems}
\author{M.~Kazarian\thanks{International Laboratory of Cluster Geometry,
HSE University and I.~M.~Krichever Center for Advanced Srudies, Skoltech},
E.~Krasilnikov\thanks{International Laboratory of Cluster Geometry,
HSE University},
S.~Lando\thanks{International Laboratory of Cluster Geometry,
HSE University and I.~M.~Krichever Center for Advanced Srudies, Skoltech},
M.~Shapiro\thanks{Michigan State University}}
\date{May 25, 2025}

\begin{document}
\maketitle

\tableofcontents

\newpage

\section{Notation}

\begin{tabular}{cp{0.8\textwidth}}
  $\fg, \langle\cdot,\cdot\rangle$ & a Lie algebra endowed with a nondegenerate invariant bilinear product\\
  $\gl(N)$ & general linear Lie algebra; consists of all $N\times N$ matrices with the commutator serving as the Lie bracket \\
  $\SL(N)$ & special linear Lie algebra; consists of all $N\times N$ trace-free matrices with the commutator serving as the Lie bracket \\
  $\so(N)$ & special orthogonal Lie algebra\\ 
  $d$ & dimension of Lie algebra; specifically, for $\gl(N)$, $d=N^2$ \\
  $E_{ij}$, & matrix units, the standard generators of the Lie algebra $\gl(N)$\\
  $1\le i,j\le N$&\\
   $F_{ij}=E_{ij}-E_{\bar j\bar i}$, & the standard generators of the Lie algebra $\so_N$\\
   $1\le i<j\le N$ &\\
  $D$ & a chord diagram \\
  $n$ & the number of chords in a chord diagram  \\
  $K_n$ & the chord diagram with $n$ chords any two of which intersect one another\\
  $\pi$ & the projection to the subspace of primitive elements 
  in a graded commutative cocommutative Hopf algebra, in particular,
  in the Hopf algebra of chord diagrams,
 the kernel of which is the subspace of decomposable elements\\
  $C_1,\cdots,C_N$ & Casimir elements in $U(\gl(N))$\\
  $w$ & a weight system\\
  $w_\fg$ & the Lie algebra weight system associated to a Lie algebra $\fg$\\
  $\bar{w}_\fg$ & $w_\fg(\pi(\cdot))$; the composition of the Lie algebra weight system $w_\fg$
  with the projection~$\pi$ to the subspace of primitives\\
  $\sigma$ & the standard cyclic permutation\\
  $\alpha$ & a permutation\\
  $m$ & the number of permuted elements; e.g. for a permutation determined by a chord diagram with~$n$ chords, $m=2n$\\
  $\lambda\vdash m$ & a partition; the cycle type of a permutation of~$m$ elements
  determines a partition of~$m$;\\
  $\nu$ & a tuple (possibly infinite) of pairwise distinct positive integers\\
  $s$ & a state of a chord diagram or a permutation;\\
  $G(\alpha)$ & the digraph of a permutation~$\alpha$\\
  $\Lambda^*(N)$ & the algebra of shifted symmetric polynomials in~$N$ variables\\
  $\phi$ & the Harish–Chandra projection\\
  $p_1,\cdots,p_N$ & shifted power sum polynomials;\\
  $\cH$ & Hopf algebra of permutations modulo generalized Vassiliev relations;\\
  $\cH_\lambda$ & homogeneous subspace in~$\cH$ determined by a partition~$\lambda$;
\end{tabular}\\

\newpage

\section{Introduction}

Chord diagrams play a crucial role in the study of finite type
knot and link invariants as introduced by V.~Vassiliev in~\cite{V90}
around 1990. Each invariant of knots of order at most~$n$
determines a weight system of order~$n$, which is a function
on chord diagrams with~$n$ chords satisfying so-called
Vassiliev's $4$-term relations. There is a variety of ways to
produce weight systems; however, these weight systems, while being easy
to define, often do not allow for efficient compu\-ta\-tions.
This is true, in particular, for powerful and rich families of
weight systems associated to Lie algebras.

A chord diagram of order~$n$ can be represented (although not uniquely) by
an arc diagram, which, in turn, can be treated as a fixed-point free involution
of $2n$ elements.
Recently, an efficient way of computing weight systems associated to Lie
algebras $\gl(N)$~\cite{ZY23,KL22}, Lie superalgebras $\gl(M|N)$~\cite{ZY23},
Lie algebras $\so(N)$ and other classical series~\cite{KY24}
has been developed. The method is based on developing recurrence relations
for computing the Lie algebra weight system by extending them to permutations;
the recurrence in question involves the values of the Lie algebra weight systems
on more general permutations even if it starts with arc diagrams.
The idea of extending certain weight systems to permutations appeared earlier
implicitly in~\cite{DK07}. In addition to making possible efficient computations,
the recurrence in question allowed one to unify all the $\gl(N)$-weight systems,
for $N=1,2,\dots$, into a single universal $\gl$-weight system, to obtain
all the $\gl(M|N)$-weight systems as a specialization of the latter,
and so on.

The above-mentioned results suggest that permutations may be considered
as generalized arc diagrams (and, under the requirement of cyclic
invariance, as generalized chord diagrams). To push this
analogy further, one should define certain analogues of Vassiliev's
$4$-term relations for permutations, which would make possible defining
weight systems for them. This is exactly what we do in the present paper.
The $\GL(N)$- and $\GL(M|N)$-weight systems
mentioned above give examples of generalized weight systems on permutations,
and we produce several new examples as well.
Vassiliev's $4$-term relations appeared naturally in the study of the geometry of
the discriminant in the space of smooth nondegenerate mappings $S^1\to S^3$.
At the moment, our understanding of generalized relations is far from being
that satisfactory: we do not know a natural geometric object they can describe.
However, we hope that such an object exists and is important.

Natural polynomial-valued 
functions on topological objects (graphs, combinatorial maps, \dots), 
after averaging prove to be solutions of integrable hierarchies
of mathematical physics. This happens to be true for the values of the
universal $\gl$-weight system on permutations as well.
We conjectured this result basing on computer experiments, and it was proved recently
by M.~Zaitsev, see Appendix.

Our second main object of concern in the present paper is various Hopf algebra 
structures on spaces of permutations. The coproduct of a chord diagram
is easily defined as a sum over all splittings of the set of its chords into
two disjoint subsets. Factoring out $4$-term relations makes it also possible
to multiply chord diagrams, which leads to a graded commutative cocommutative
Hopf algebra structure on the space of chord diagrams. This structure 
not only allows one to compare the strength of different knot invariants and
weight systems, but also to separate the properties of weight systems determined
by the Hopf algebra structure from those following from their values on primitive elements. It happens that many weight systems behave naturally with respect to
this structure. Until recently, only rather rough properties of Lie algebra
weight systems were investigated in this way. In previous papers of the authors
of the present one with coauthors, much more fine properties have been established.

Similarly to the case of chord diagrams, factoring out the generalized
Vassiliev's relations also allows one to introduce a Hopf algebra structure
on the space of permutations. However, considering the universal $\gl$- and
$\so$-weight systems suggests another, simpler construction of Hopf algebras,
which does not involve relations of Vassiliev type.
These are generated by cyclic equivalence classes of connected 
permutations, and we call them rotational Hopf algebras. 

The paper is organized as follows. In Sec.~\ref{s3}, we introduce generalized
Vassiliev's relation for permutations and give examples of functions on permutations
satisfying these relations, which are based on the Lie algebra weight systems.

In Sec.~\ref{s4}, we recall the construction of the universal $\gl$-
and $\so$-weight systems on permutations, relate the standard substitution 
for these Lie algebras with the topology of the corresponding hypermaps.
We deduce a formula, not known earlier, 
for the value of the universal $\gl$-weight system
on the inverse permutation $\alpha^{-1}$ provided its value 
on~$\alpha$ is known; this formula is based on a duality in Lie
algebras $\gl(N)$.
We also state a theorem relating the average value of the $\gl$-
and $\so$-weight systems on
permutations to the Kadomtsev--Petviashvili hierarchy of partial differential equations.

Sec.~\ref{s5} is devoted to the study of Hopf algebra structures on 
spaces of permutations. In particular, we describe a rescaling for the
$\gl$-weight system, which makes its leading term a Hopf algebra homomorphism
from the Hopf algebra of positive permutations to that of polynomials
in infinitely many variables~\cite{KKL24}.

The theorem relating the average value of the $\gl$-weight systems on
permutations to the Kadomtsev--Petviashvili hierarchy is proved in
a separate paper by M.~Zaitsev, to appear.

M.K., E.K., and S.L. are  partially supported by the RSF Grant 24-11-00366
“Singularity theory and integrability”.

\section{Generalized Vassiliev relations}\label{s3}

In 1990, V.~Vassiliev~\cite{V90} introduced the notion of finite type invariant of
knots. He associated to each knot invariant of order at most~$n$ a function on chord
diagrams with~$n$ chords. He has shown that each such function satisfies
so-called $4$-term relations. In~1993, M.~Kontsevich proved the converse:
if a function on chord diagrams with~$n$ chords taking values in a field
of characteristic~$0$ satisfies $4$-term relations, then it can be obtained
by Vassiliev's construction from some finite type knot invariant of order
at most~$n$. (Both in Vassiliev's and Kontsevich's theorems the function
must satisfy an additional requirement, the so-called one-term relation;
this requirement, however, can be eliminated by replacing ordinary knots
with framed ones, and does not affect essentially the construction of
the theory, and we will not mention it below).
Functions on chord diagrams satisfying $4$-term relations are called \emph{weight systems}.

Below, in Sec.~\ref{ss31} we recall the notions of chord diagrams and $4$-term
relations. In recent papers~\cite{KL22,ZY22,KY24}, weight systems associated 
to classical series of Lie algebras have been extended to objects more
general than chord diagrams, namely, to permutations considered
modulo cyclic shifts. The main goal of this extension consists in obtaining
effective recurrence relations for computing Lie algebra weight systems.
However, defining weight systems on permutations suggests that the latter
are valuable by themselves. In Sec.~\ref{ss32} we introduce generalized
Vassiliev relations for them and show that the $\gl$- and $\so$-weight systems
satisfy these relations. Currently, we do not know a geometric object 
(similar to knots in the Vassiliev case), which is in charge of these
new relations.

For chord diagrams, the generalized Vassiliev relations reduce to the original
$4$-term ones.
However, in contrast to original Vassiliev's relations,
for general permutations there are
two types of generalized relations, and the number of terms in such a relation
depends on the permutations and specific disjoint cycles in them to which
we do apply it.

\subsection{Chord diagrams and Vassiliev's $4$-term relations}\label{ss31}

A \emph{chord diagram of order~$n$} is a pair consisting of an oriented circle
and $2n$ pairwise distinct points on it split into~$n$ disjoint pairs considered
up to orientation preserving diffeomorphisms of the circle.
The circle under consideration is also called the \emph{Wilson loop}
of the chord diagram.

In the pictures below, we always assume that the circle is oriented counterclockwise.
The points of the circle belonging to the same pair are connected by a line or arc
segment.

A $4$-term relation is defined by a chord diagram and a pair of chords
in it having neighboring ends. We say that a
 \emph{function~$f$ on chord diagrams satisfies Vassiliev's $4$-term relations}
if for any chord diagram and any pair of chords $a,b$
in it having neighboring ends the equation shown in Fig.~\ref{fourtermrelation} holds.

\begin{figure}[ht]
\[
f\left(\begin{tikzpicture}[baseline={([yshift=-.5ex]current bounding box.center)}]
	\draw[dashed] (0,0) circle (1);
	\filldraw[black] (-.8,0)  node[anchor=west]{$a$};
	\filldraw[black] (.3,0)  node[anchor=west]{$b$};
	\draw[line width=1pt]  ([shift=( 20:1cm)]0,0) arc [start angle= 20, end angle= 70, radius=1];
	\draw[line width=1pt]  ([shift=(110:1cm)]0,0) arc [start angle=110, end angle=160, radius=1];
	\draw[line width=1pt]  ([shift=(250:1cm)]0,0) arc [start angle=250, end angle=290, radius=1];
	\draw[line width=1pt] (45:1) ..  controls (5:0.3) and (-40:0.3)  .. (280:1);
	\draw[line width=1pt] (135:1) ..  controls (175:0.3) and (220:0.3)  .. (260:1);
\end{tikzpicture}\right)  -
f\left(\begin{tikzpicture}[baseline={([yshift=-.5ex]current bounding box.center)}]
	\filldraw[black] (-.6,0)  node[anchor=west]{$a$};
	\filldraw[black] (.2,0)  node[anchor=west]{$b$};
	\draw[dashed] (0,0) circle (1);
	\draw[line width=1pt]  ([shift=( 20:1cm)]0,0) arc [start angle= 20, end angle= 70, radius=1];
	\draw[line width=1pt]  ([shift=(110:1cm)]0,0) arc [start angle=110, end angle=160, radius=1];
	\draw[line width=1pt]  ([shift=(250:1cm)]0,0) arc [start angle=250, end angle=290, radius=1];
	\draw[line width=1pt] (45:1) ..  controls (-5:0.1) and (-50:0.1)  .. (260:1);
	\draw[line width=1pt] (135:1) ..  controls (185:0.1) and (225:0.1)  .. (280:1);
\end{tikzpicture}\right)  =
f\left(\begin{tikzpicture}[baseline={([yshift=-.5ex]current bounding box.center)}]
	\filldraw[black] (-.4,.4)  node[anchor=west]{$a$};
	\filldraw[black] (.4,0)  node[anchor=west]{$b$};
	\draw[dashed] (0,0) circle (1);
	\draw[line width=1pt]  ([shift=( 20:1cm)]0,0) arc [start angle= 20, end angle= 70, radius=1];
	\draw[line width=1pt]  ([shift=(110:1cm)]0,0) arc [start angle=110, end angle=160, radius=1];
	\draw[line width=1pt]  ([shift=(250:1cm)]0,0) arc [start angle=250, end angle=290, radius=1];
	\draw[line width=1pt] (35:1) ..  controls (0:0.3) and (-45:0.3)  .. (280:1);
	\draw[line width=1pt] (135:1) ..  controls (105:0.5) and (85:0.5)  .. (55:1);
\end{tikzpicture}\right)  -
f\left(\begin{tikzpicture}[baseline={([yshift=-.5ex]current bounding box.center)}]
    \filldraw[black] (-.4,.3)  node[anchor=west]{$a$};
	\filldraw[black] (.2,0)  node[anchor=west]{$b$};
	\draw[dashed] (0,0) circle (1);
	\draw[line width=1pt]  ([shift=( 20:1cm)]0,0) arc [start angle= 20, end angle= 70, radius=1];
	\draw[line width=1pt]  ([shift=(110:1cm)]0,0) arc [start angle=110, end angle=160, radius=1];
	\draw[line width=1pt]  ([shift=(250:1cm)]0,0) arc [start angle=250, end angle=290, radius=1];
	\draw[line width=1pt] (55:1) ..  controls (5:0.1) and (-40:0.1)  .. (270:1);
	\draw[line width=1pt] (135:1) ..  controls (105:0.4) and (65:0.4)  .. (35:1);
\end{tikzpicture}\right)
\]
\caption{$4$-term relations for chord diagrams}
\label{fourtermrelation}
\end{figure}

The relation in the picture has the following meaning.
The chord diagrams entering it can have an additional
set of chords, whose ends belong to the dotted arcs,
and which is one and the same for all the diagrams in the relation.
Out of the $4$ ends of the two distinguished chord $a,b$,
both ends of the chord~$b$ are fixed, as well as one of the ends
of the chord~$a$, while the second end of~$a$ takes
successively all the four possible positions close to the ends of~$b$;
the position of the second end of~$a$ determines the term
of the relation uniquely. Note that it is not important
whether the chords~$a$ and~$b$ in the leftmost chord diagram
intersect one another: in order to replace nonintersecting pair
with an intersecting one it suffices to multiply the relation in
Fig.~~\ref{fourtermrelation} by~$-1$.

An \emph{arc diagram} is a representation of a chord diagram, 
in which the chords'  ends are placed along an oriented line with edges drawn as semicircles. 
An arc diagram is obtained from a chord diagram by cutting
the Wilson loop at an arbitrary point distinct from the ends
of the chords and developing it into a horizontal line,
oriented to the right. Depending on the choice of the cut point,
up to~$2n$ distinct arc diagrams can be associated to a given chord diagram, while an arc diagram determines the corresponding chord diagram 
in a unique way.
An example of a chord diagram and a 
corresponding arc diagram is depicted in Fig.~\ref{SimpleArcChordDiagrams}.

\begin{figure}[ht]
	\[
	\begin{tikzpicture}[baseline={(current bounding box.center)}, scale=0.5]
	//======================================== 1
	\pgfmathsetmacro\x{1}
	\pgfmathsetmacro\y{-1}
	\pgfmathsetmacro\n{8}
	\draw[gray, thick] (\x,\y) -- (\x+\n+1,\y);
	\foreach \i in {1,...,\n} {
		\node (n\i) at (\x+\i, \y) {};
	};
	\foreach \t/\u in {1/3, 2/7, 4/8, 5/6} {
		\pgfmathsetmacro\r{(abs(\u-\t))/2}
		\ifthenelse 
		{\t < \u}
		{\draw [thick] (\x+\u,\y) arc (0:90:\r);
			\draw [thick] (\x+\t+\r,\y+\r) arc (90:180:\r);
		}
		{\draw [thick] (\x+\u,\y) arc (180:270:\r);
			\draw [thick, stealth-] (\x+\u+\r,\y-\r) arc (270:360:\r);
		}
	};
 	\foreach \i in {1,...,\n} {
		\node[shape=circle,fill=gray, scale=0.5] at (\x+\i, \y) {};
		\node [below] at (\x+\i, \y) {\i};
	};
	//======================================== 2
	\pgfmathsetmacro\x{15}
	\pgfmathsetmacro\y{0}
	\pgfmathsetmacro\rad{2}
	\pgfmathsetmacro\n{8}
	\draw [gray, thick] (0+\x,0+\y) circle (\rad);
	\foreach \i in {1,...,\n} {
		\pgfmathsetmacro\r{\i*(360/\n)}
		\node at ($ (\r:\rad+0.6) + (\x,\y) $) {\i};
		\node[shape=circle,fill=gray, scale=0.5] (n\i) at ($ (\r:\rad) + (\x,\y) $) {};
	};
	\foreach \t/\u in {1/3, 2/7, 4/8, 5/6} {
		\draw [thick] (n\t) -- (n\u);
	};
	\end{tikzpicture} 
	\]
	\caption{An arc diagram and the corresponding chord diagram}
	\label{SimpleArcChordDiagrams}
\end{figure}

\subsection{Hyper chord and hyper arc diagrams}\label{ss32}
An arc diagram can be interpreted as a pair $(\sigma,\alpha)$
of permutations acting on a finite set in the following way.
If the diagram has~$n$ arcs, then the $2n$ ends of the arcs can be
numbered from $1$ to~$2n$ in the increasing order, going from left to right.
The first permutation $\sigma$ is the standard cycle of length~$2n$,
$\sigma=(1,2,3,\dots,2n)$, while the second one, namely $\alpha$, is
the fixed point free involution on the set of arc ends, which exchanges the
ends of each arc. The chord diagram corresponding to a given arc diagram
is a pair of permutations $(\sigma,\alpha)$ considered up to a cyclic
shift of the sequence $(1,2,3,\dots,2n)$ (which, in particular,
preserves the permutation~$\sigma$). Below, we will
use the same notation $(\sigma,\alpha)$
to denote the corresponding chord diagram.

In~\cite{ZY22}, it was suggested  to extend certain weight systems to
more general pairs of permutations $(\sigma,\alpha)$. Namely,
$\sigma$ still is the standard long cycle $(1,2,3,\dots,m)\in \BS_m$,
while $\alpha\in \BS_m$ is an arbitrary permutation,
not necessarily an involution without fixed points.
We call such a pair a {\emph{hyper arc diagram} and we call a pair $(\sigma,\alpha)$ considered up to a cyclic permutation} a \emph{hyper chord diagram};
the terminology will be explained later.

By a \emph{hyper edge} of a hyper chord diagram $(\sigma,\alpha)$ we mean a cycle
in the decomposition of~$\alpha$ into the product of disjoint cycles.
The elements of the cycle will be called the \emph{legs} of the hyper edge
(whence, a hyper edge can have $1,2,3,\dots$ legs). A hyper edge with~$k$ legs
will also be called a $k$-\emph{edge}, so that a $2$-edge
is just an ordinary edge. In the figures below, a permutation~$\alpha$
is depicted as its directed graph $G(\alpha)$, in which
the permuted elements follow the horizontal line in the
increasing order and each element~$i$ is connected to
its image $\alpha(i)$ by an arc oriented from~$i$ to $\alpha(i)$.
In some cases, the numbering of the permuted elements will be omitted.

Figure ~\ref{HyperChordDiagrams} shows an example of a hyper chord diagram and all the five corresponding hyper arc diagrams,
each of which has one $2$-edge and one $3$-edge.


\begin{figure}[ht]
	\[
	\begin{tikzpicture}[baseline={(current bounding box.center)}, scale=0.5]
	//======================================== 1
	\pgfmathsetmacro\x{1}
	\pgfmathsetmacro\y{0}
	\pgfmathsetmacro\rad{2}
	\pgfmathsetmacro\n{5}
	\draw [gray, thick] (0+\x,0+\y) circle (\rad);
	\foreach \i in {1,...,\n} {
		\pgfmathsetmacro\r{\i*(360/\n)}
		\node[shape=circle,fill=gray, scale=0.5] (n\i) at ($ (\r:\rad) + (\x,\y) $) {};
	};
	\foreach \t/\u in {1/3, 2/4, 3/1, 4/5, 5/2} {
		\draw [thick, -stealth] (n\t) -- (n\u);
	}
	//======================================== 2
	\pgfmathsetmacro\x{5}
	\pgfmathsetmacro\y{2}
	\pgfmathsetmacro\n{5}
	\draw[gray, thick] (\x,\y) -- (\x+\n+1,\y);
	\foreach \i in {1,...,\n} {
		\node (n\i) at (\x+\i, \y) {};
	};
	\foreach \t/\u in {1/3, 3/1, 2/4, 4/5, 5/2} {
		\pgfmathsetmacro\r{(abs(\u-\t))/2}
		\ifthenelse 
		{\t < \u}
		{\draw [thick] (\x+\u,\y) arc (0:90:\r);
			\draw [thick, stealth-] (\x+\t+\r,\y+\r) arc (90:180:\r);
		}
		{\draw [thick] (\x+\u,\y) arc (180:270:\r);
			\draw [thick, stealth-] (\x+\u+\r,\y-\r) arc (270:360:\r);
		}
	}
	\foreach \i in {1,...,\n} {
		\node[shape=circle,fill=gray, scale=0.5] (n\i) at (\x+\i, \y) {};
		\node [below left] at (\x+\i, \y) {\i};
	};
	//======================================== 3
	\pgfmathsetmacro\x{5}
	\pgfmathsetmacro\y{-2}
	\pgfmathsetmacro\n{5}
	\draw[gray, thick] (\x,\y) -- (\x+\n+1,\y);
	\foreach \i in {1,...,\n} {
		\node[shape=circle,fill=gray, scale=0.5] (n\i) at (\x+\i, \y) {};
		\node [above left] at (\x+\i, \y) {\i};
	};
	\foreach \t/\u in {1/4, 4/2, 2/1, 3/5, 5/3} {
		\pgfmathsetmacro\r{(abs(\u-\t))/2}
		\ifthenelse 
		{\t < \u}
		{\draw [thick] (\x+\u,\y) arc (0:90:\r);
			\draw [thick, stealth-] (\x+\t+\r,\y+\r) arc (90:180:\r);
		}
		{\draw [thick] (\x+\u,\y) arc (180:270:\r);
			\draw [thick, stealth-] (\x+\u+\r,\y-\r) arc (270:360:\r);
		}
	}
	\foreach \i in {1,...,\n} {
		\node[shape=circle,fill=gray, scale=0.5] (n\i) at (\x+\i, \y) {};
		\node [above left] at (\x+\i, \y) {\i};
	};
	//======================================== 4
	\pgfmathsetmacro\x{12}
	\pgfmathsetmacro\y{2}
	\pgfmathsetmacro\n{5}
	\draw[gray, thick] (\x,\y) -- (\x+\n+1,\y);
	\foreach \i in {1,...,\n} {
		\node (n\i) at (\x+\i, \y) {};
	};
	\foreach \t/\u in {1/4, 4/1, 2/5, 5/3, 3/2} {
		\pgfmathsetmacro\r{(abs(\u-\t))/2}
		\ifthenelse 
		{\t < \u}
		{\draw [thick] (\x+\u,\y) arc (0:90:\r);
			\draw [thick, stealth-] (\x+\t+\r,\y+\r) arc (90:180:\r);
		}
		{\draw [thick] (\x+\u,\y) arc (180:270:\r);
			\draw [thick, stealth-] (\x+\u+\r,\y-\r) arc (270:360:\r);
		}
	}
	\foreach \i in {1,...,\n} {
		\node[shape=circle,fill=gray, scale=0.5] (n\i) at (\x+\i, \y) {};
		\node [above right] at (\x+\i, \y) {\i};
	};
	//======================================== 5
	\pgfmathsetmacro\x{12}
	\pgfmathsetmacro\y{-2}
	\pgfmathsetmacro\n{5}
	\draw[gray, thick] (\x,\y) -- (\x+\n+1,\y);
	\foreach \i in {1,...,\n} {
		\node (n\i) at (\x+\i, \y) {};
	};
	\foreach \t/\u in {1/4, 4/3, 3/1, 2/5, 5/2} {
		\pgfmathsetmacro\r{(abs(\u-\t))/2}
		\ifthenelse 
		{\t < \u}
		{\draw [thick] (\x+\u,\y) arc (0:90:\r);
			\draw [thick, stealth-] (\x+\t+\r,\y+\r) arc (90:180:\r);
		}
		{\draw [thick] (\x+\u,\y) arc (180:270:\r);
			\draw [thick, stealth-] (\x+\u+\r,\y-\r) arc (270:360:\r);
		}
	}
	\foreach \i in {1,...,\n} {
		\node[shape=circle,fill=gray, scale=0.5] (n\i) at (\x+\i, \y) {};
		\node [above right] at (\x+\i, \y) {\i};
	};
	//======================================== 6
	\pgfmathsetmacro\x{19}
	\pgfmathsetmacro\y{0}
	\pgfmathsetmacro\n{5}
	\draw[gray, thick] (\x,\y) -- (\x+\n+1,\y);
	\foreach \i in {1,...,\n} {
		\node (n\i) at (\x+\i, \y) {};
	};
	\foreach \t/\u in {1/5, 5/3, 3/1, 2/4, 4/2} {
		\pgfmathsetmacro\r{(abs(\u-\t))/2}
		\ifthenelse 
		{\t < \u}
		{\draw [thick] (\x+\u,\y) arc (0:90:\r);
			\draw [thick, stealth-] (\x+\t+\r,\y+\r) arc (90:180:\r);
		}
		{\draw [thick] (\x+\u,\y) arc (180:270:\r);
			\draw [thick, stealth-] (\x+\u+\r,\y-\r) arc (270:360:\r);
		}
	}
	\foreach \i in {1,...,\n} {
		\node[shape=circle,fill=gray, scale=0.5] (n\i) at (\x+\i, \y) {};
		\node [below left] at (\x+\i, \y) {\i};
	};
	//========================================
	\end{tikzpicture} 
	\]
	\caption{A hyper chord diagram and the corresponding hyper arc diagrams; the hyper $2$-edge
    of the hyper chord diagram is depicted as a double-headed arrow}
	\label{HyperChordDiagrams}
\end{figure}
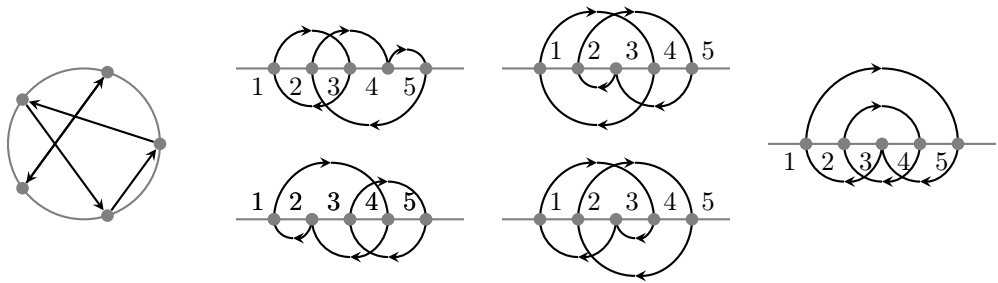

\newpage

An ordinary chord diagram is a hyper chord diagram
all whose hyper edges are $2$-edges.

To each chord diagram one can associate an oriented surface with boundary
in the following way: attach a disk to the Wilson loop and consider it
as the vertex of the map, and make each chord into a ribbon by thickening it.
Similarly, to each hyper chord diagram one also associates an oriented
surface with boundary:
\begin{itemize}
\item attach a disk to the Wilson loop (the vertex) and mark~$m$
disjoint arc segments along the Wilson loop, numbered $1,2,\dots,m$
in the successive order;
\item associate a disk to each hyper edge $(n_1,n_2,\dots,n_k)$  and mark
$k$ pairwise nonintersecting
arc segments of its boundary in the successive order;
\item
attach this disk to the vertex along the $k$ segments having the same numbers in such a way that the end of the incoming arrow at each point precedes that of the outgoing one.
\end{itemize}
The reader may find details below in Sec.~\ref{ss42}.

\subsection{Generalized Vassiliev relations for permutations}\label{ssgVr}
For hyper arc diagrams,
we introduce two types of \emph{generalized Vassiliev relations}.

\begin{definition}

A \emph{one-hyper-arc Vassiliev relation} equates to zero the alternating sum 
of $2(\ell-1)$ hyper arc diagrams, in which all the hyper arcs but one
are the same, while $\ell-1$ legs of the last hyper arc, of length~$\ell$,
are fixed, and the $\ell$~th leg 
(called the \emph{free leg}) takes all the  $2(\ell-1)$ positions next
to the fixed $\ell-1$ legs; a hyper arc diagram enters the sum
with sign~$+$ if the free leg is before the corresponding fixed leg,
and sign~$-$, if its position is after the leg. We call such alternating sum a \emph{one-hyper-arc element}.
\end{definition}

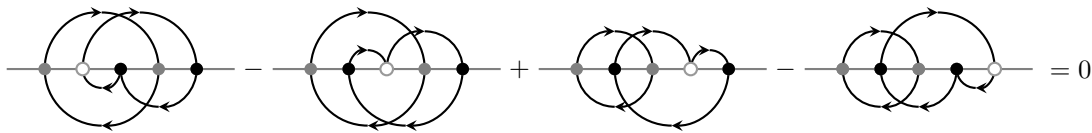
\begin{figure}[ht]
		\[
		\begin{tikzpicture}[baseline={(current bounding box.center)}, scale=0.5]
		//======================================== 1
		\pgfmathsetmacro\x{0}
		\pgfmathsetmacro\y{0}
		\pgfmathsetmacro\n{5}
  		\draw[gray, thick] (\x,\y) -- (\x+\n+1,\y);
		\foreach \i in {1,...,\n} {
			\node (n\i) at (\x+\i, \y) {};
		};
		\foreach \t/\u in {1/4, 4/1, 2/5, 5/3, 3/2} {
			\pgfmathsetmacro\r{(abs(\u-\t))/2}
			\ifthenelse 
			{\t < \u}
			{\draw [thick] (\x+\u,\y) arc (0:90:\r);
				\draw [thick, stealth-] (\x+\t+\r,\y+\r) arc (90:180:\r);
			}
			{\draw [thick] (\x+\u,\y) arc (180:270:\r);
				\draw [thick, stealth-] (\x+\u+\r,\y-\r) arc (270:360:\r);
			}
		}
		\foreach \i in {1,...,\n} {
			\node[shape=circle,fill=gray, scale=0.5] (n\i) at (\x+\i, \y) {};
		};
		\foreach \i in {3, 5} {
			\node[shape=circle,fill=black, scale=0.5] (n\i) at (\x+\i, \y) {};
		};
		\node[shape=circle, draw=gray!80, thick, fill=white, scale=0.5] (n0) at (\x+2, \y) {};
		//======================================== 2
		\node at (6.5, 0) {$-$};
		//========================================
		\pgfmathsetmacro\x{7}
		\pgfmathsetmacro\y{0}
		\pgfmathsetmacro\n{5}
  		\draw[gray, thick] (\x,\y) -- (\x+\n+1,\y);
		\foreach \i in {1,...,\n} {
			\node (n\i) at (\x+\i, \y) {};
		};
		\foreach \t/\u in {1/4, 4/1, 2/3, 3/5, 5/2} {
			\pgfmathsetmacro\r{(abs(\u-\t))/2}
			\ifthenelse 
			{\t < \u}
			{\draw [thick] (\x+\u,\y) arc (0:90:\r);
				\draw [thick, stealth-] (\x+\t+\r,\y+\r) arc (90:180:\r);
			}
			{\draw [thick] (\x+\u,\y) arc (180:270:\r);
				\draw [thick, stealth-] (\x+\u+\r,\y-\r) arc (270:360:\r);
			}
		}
		\foreach \i in {1,...,\n} {
			\node[shape=circle,fill=gray, scale=0.5] (n\i) at (\x+\i, \y) {};
		};
		\foreach \i in {2, 5} {
			\node[shape=circle,fill=black, scale=0.5] (n\i) at (\x+\i, \y) {};
		};
		\node[shape=circle, draw=gray!80, thick, fill=white, scale=0.5] (n0) at (\x+3, \y) {};
		//======================================== 3
		\node at (13.5, 0) {$+$};
		//========================================
		\pgfmathsetmacro\x{14}
		\pgfmathsetmacro\y{0}
		\pgfmathsetmacro\n{5}
		\draw[gray, thick] (\x,\y) -- (\x+\n+1,\y);
		\foreach \i in {1,...,\n} {
			\node (n\i) at (\x+\i, \y) {};
		};
		\foreach \t/\u in {1/3, 3/1, 2/4, 4/5, 5/2} {
			\pgfmathsetmacro\r{(abs(\u-\t))/2}
			\ifthenelse 
			{\t < \u}
			{\draw [thick] (\x+\u,\y) arc (0:90:\r);
				\draw [thick, stealth-] (\x+\t+\r,\y+\r) arc (90:180:\r);
			}
			{\draw [thick] (\x+\u,\y) arc (180:270:\r);
				\draw [thick, stealth-] (\x+\u+\r,\y-\r) arc (270:360:\r);
			}
		}
		\foreach \i in {1,...,\n} {
			\node[shape=circle,fill=gray, scale=0.5] (n\i) at (\x+\i, \y) {};
		};
		\foreach \i in {2, 5} {
			\node[shape=circle,fill=black, scale=0.5] (n\i) at (\x+\i, \y) {};
		};
		\node[shape=circle, draw=gray!80, thick, fill=white, scale=0.5] (n0) at (\x+4, \y) {};
		//======================================== 4
		\node at (20.5, 0) {$-$};
		//========================================
		\pgfmathsetmacro\x{21}
		\pgfmathsetmacro\y{0}
		\pgfmathsetmacro\n{5}
		\draw[gray, thick] (\x,\y) -- (\x+\n+1,\y);
		\foreach \i in {1,...,\n} {
			\node[shape=circle,fill=gray, scale=0.5] (n\i) at (\x+\i, \y) {};
		};
		\foreach \t/\u in {1/3, 3/1, 2/5, 5/4, 4/2} {
			\pgfmathsetmacro\r{(abs(\u-\t))/2}
			\ifthenelse 
			{\t < \u}
			{\draw [thick] (\x+\u,\y) arc (0:90:\r);
				\draw [thick, stealth-] (\x+\t+\r,\y+\r) arc (90:180:\r);
			}
			{\draw [thick] (\x+\u,\y) arc (180:270:\r);
				\draw [thick, stealth-] (\x+\u+\r,\y-\r) arc (270:360:\r);
			}
		}
		\foreach \i in {1,...,\n} {
			\node[shape=circle,fill=gray, scale=0.5] (n\i) at (\x+\i, \y) {};
		};
		\foreach \i in {2, 4} {
			\node[shape=circle,fill=black, scale=0.5] (n\i) at (\x+\i, \y) {};
		};
		\node[shape=circle, draw=gray!80, thick, fill=white, scale=0.5] (n0) at (\x+5, \y) {};
		//======================================== 
		\node at (28, 0) {$ = 0$};
		//========================================
		\end{tikzpicture} 
		\]
		\caption{One-hyper-arc Vassiliev relation; the free leg is depicted as a white disc, while the fixed legs of the chosen hyper edge are shown in black}
		\label{OneHyperArcVassilievRelation}
\end{figure}

Note that for chord diagrams, one-hyper-arc Vassiliev relations are empty.
 
\begin{definition}
A \emph{two-hyper-arc Vassiliev relation} equates to zero the alternating sum 
of $2\ell$ hyper arc diagrams, in which all the hyper arcs but one
are the same, while $k-1$ legs of the last hyper arc, of length~$k$,
are fixed, and the $k$~th leg 
(called the \emph{free leg}) takes all the  $2\ell$ positions next
to the $\ell$ legs of a fixed {hyper arc} of length~$\ell$; a hyper arc diagram enters the sum
with sign~$+$ if the free leg is before the corresponding fixed leg,
and sign~$-$ if its position is after the leg. We call such alternating sum a \emph{two-hyper-arc element}.
\end{definition}

We say that a function on permutations \emph{satisfies generalized
Vassiliev relations} if 
for each one-hyper-arc and each two-hyper-arc element
the alternating sum of its values is zero.
For ordinary arc diagrams, two-hyper-arc generalized Vassiliev  relations coincide with 
ordinary $4$-term relations.


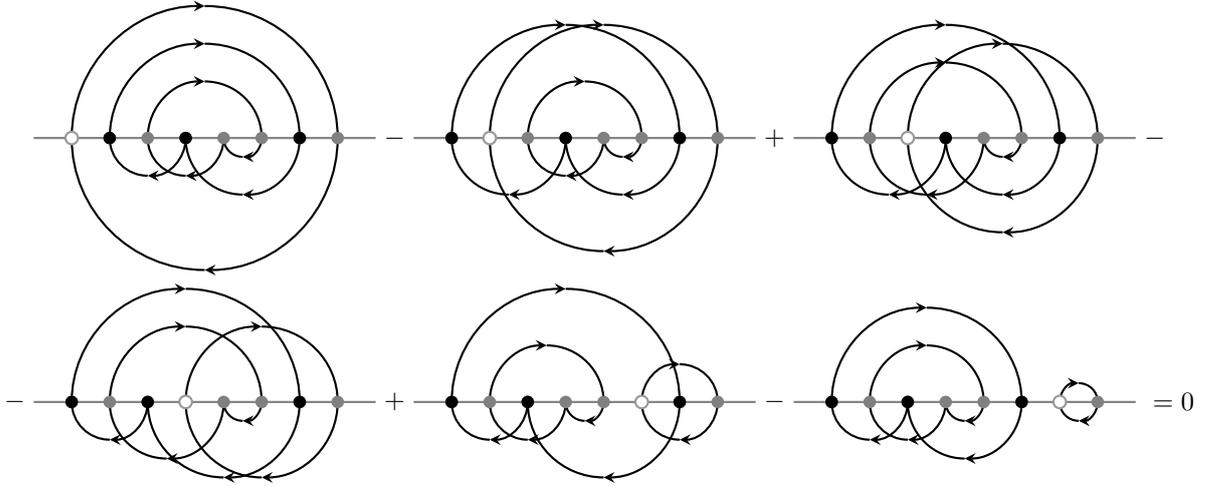
\begin{figure}[ht]
	\[
	\begin{tikzpicture}[baseline={(current bounding box.center)}, scale=0.5]
	//======================================== 1
	\pgfmathsetmacro\x{0}
	\pgfmathsetmacro\y{0}
	\pgfmathsetmacro\n{8}
	\draw[gray, thick] (\x,\y) -- (\x+\n+1,\y);
	\foreach \i in {1,...,\n} {
		\node (n\i) at (\x+\i, \y) {};
	};
	\foreach \t/\u in {1/8, 8/1, 2/7, 7/4, 4/2, 3/6, 6/5, 5/3} {
		\pgfmathsetmacro\r{(abs(\u-\t))/2}
		\ifthenelse 
		{\t < \u}
		{\draw [thick] (\x+\u,\y) arc (0:90:\r);
			\draw [thick, stealth-] (\x+\t+\r,\y+\r) arc (90:180:\r);
		}
		{\draw [thick] (\x+\u,\y) arc (180:270:\r);
			\draw [thick, stealth-] (\x+\u+\r,\y-\r) arc (270:360:\r);
		}
	}
 	\foreach \i in {1,...,\n} {
		\node[shape=circle,fill=gray, scale=0.5] (n\i) at (\x+\i, \y) {};
	};
  	\foreach \i in {2, 4, 7} {
		\node[shape=circle,fill=black, scale=0.5] (n\i) at (\x+\i, \y) {};
	};
	\node[shape=circle, draw=gray!80, thick, fill=white, scale=0.5] (n0) at (\x+1, \y) {};
	//======================================== 2
	\node at (8.5+1, 0) {$-$};
	//========================================
	\pgfmathsetmacro\x{9+1}
	\pgfmathsetmacro\y{0}
	\pgfmathsetmacro\n{8}
	\draw[gray, thick] (\x,\y) -- (\x+\n+1,\y);
	\foreach \i in {1,...,\n} {
		\node (n\i) at (\x+\i, \y) {};
	};
	\foreach \t/\u in {1/7, 7/4, 4/1, 2/8, 8/2, 3/6, 6/5, 5/3} {
		\pgfmathsetmacro\r{(abs(\u-\t))/2}
		\ifthenelse 
		{\t < \u}
		{\draw [thick] (\x+\u,\y) arc (0:90:\r);
			\draw [thick, stealth-] (\x+\t+\r,\y+\r) arc (90:180:\r);
		}
		{\draw [thick] (\x+\u,\y) arc (180:270:\r);
			\draw [thick, stealth-] (\x+\u+\r,\y-\r) arc (270:360:\r);
		}
	}
 	\foreach \i in {1,...,\n} {
		\node[shape=circle,fill=gray, scale=0.5] (n\i) at (\x+\i, \y) {};
	};
  	\foreach \i in {1, 4, 7} {
		\node[shape=circle,fill=black, scale=0.5] (n\i) at (\x+\i, \y) {};
	};
	\node[shape=circle, draw=gray!80, thick, fill=white, scale=0.5] (n0) at (\x+2, \y) {};
	//======================================== 3
	\node at (17.5+2, 0) {$+$};
	//========================================
	\pgfmathsetmacro\x{18+2}
	\pgfmathsetmacro\y{0}
	\pgfmathsetmacro\n{8}
	\draw[gray, thick] (\x,\y) -- (\x+\n+1,\y);
	\foreach \i in {1,...,\n} {
		\node (n\i) at (\x+\i, \y) {};
	};
	\foreach \t/\u in {1/7, 7/4, 4/1, 2/6, 6/5, 5/2, 3/8, 8/3} {
		\pgfmathsetmacro\r{(abs(\u-\t))/2}
		\ifthenelse 
		{\t < \u}
		{\draw [thick] (\x+\u,\y) arc (0:90:\r);
			\draw [thick, stealth-] (\x+\t+\r,\y+\r) arc (90:180:\r);
		}
		{\draw [thick] (\x+\u,\y) arc (180:270:\r);
			\draw [thick, stealth-] (\x+\u+\r,\y-\r) arc (270:360:\r);
		}
	}
 	\foreach \i in {1,...,\n} {
		\node[shape=circle,fill=gray, scale=0.5] (n\i) at (\x+\i, \y) {};
	};
  	\foreach \i in {1, 4, 7} {
		\node[shape=circle,fill=black, scale=0.5] (n\i) at (\x+\i, \y) {};
	};
	\node[shape=circle, draw=gray!80, thick, fill=white, scale=0.5] (n0) at (\x+3, \y) {};
	//======================================== 4
	\node at (26.5+3, 0) {$-$};
	\node at (-0.5, -7) {$-$};
	//========================================
	\pgfmathsetmacro\x{0}
	\pgfmathsetmacro\y{-7}
	\pgfmathsetmacro\n{8}
	\draw[gray, thick] (\x,\y) -- (\x+\n+1,\y);
	\foreach \i in {1,...,\n} {
		\node (n\i) at (\x+\i, \y) {};
	};
	\foreach \t/\u in {1/7, 7/3, 3/1, 2/6, 6/5, 5/2, 4/8, 8/4} {
		\pgfmathsetmacro\r{(abs(\u-\t))/2}
		\ifthenelse 
		{\t < \u}
		{\draw [thick] (\x+\u,\y) arc (0:90:\r);
			\draw [thick, stealth-] (\x+\t+\r,\y+\r) arc (90:180:\r);
		}
		{\draw [thick] (\x+\u,\y) arc (180:270:\r);
			\draw [thick, stealth-] (\x+\u+\r,\y-\r) arc (270:360:\r);
		}
	}
 	\foreach \i in {1,...,\n} {
		\node[shape=circle,fill=gray, scale=0.5] (n\i) at (\x+\i, \y) {};
	};
  	\foreach \i in {1, 3, 7} {
		\node[shape=circle,fill=black, scale=0.5] (n\i) at (\x+\i, \y) {};
	};
	\node[shape=circle, draw=gray!80, thick, fill=white, scale=0.5] (n0) at (\x+4, \y) {};
	//======================================== 5
	\node at (8.5+1, -7) {$+$};
	//========================================
	\pgfmathsetmacro\x{9+1}
	\pgfmathsetmacro\y{-7}
	\pgfmathsetmacro\n{8}
	\draw[gray, thick] (\x,\y) -- (\x+\n+1,\y);
	\foreach \i in {1,...,\n} {
		\node (n\i) at (\x+\i, \y) {};
	};
	\foreach \t/\u in {1/7, 7/3, 3/1, 2/5, 5/4, 4/2, 6/8, 8/6} {
		\pgfmathsetmacro\r{(abs(\u-\t))/2}
		\ifthenelse 
		{\t < \u}
		{\draw [thick] (\x+\u,\y) arc (0:90:\r);
			\draw [thick, stealth-] (\x+\t+\r,\y+\r) arc (90:180:\r);
		}
		{\draw [thick] (\x+\u,\y) arc (180:270:\r);
			\draw [thick, stealth-] (\x+\u+\r,\y-\r) arc (270:360:\r);
		}
	}
 	\foreach \i in {1,...,\n} {
		\node[shape=circle,fill=gray, scale=0.5] (n\i) at (\x+\i, \y) {};
	};
  	\foreach \i in {1, 3, 7} {
		\node[shape=circle,fill=black, scale=0.5] (n\i) at (\x+\i, \y) {};
	};
	\node[shape=circle, draw=gray!80, thick, fill=white, scale=0.5] (n0) at (\x+5+1, \y) {};
	//======================================== 6
	\node at (17.5+2, -7) {$-$};
	//========================================
	\pgfmathsetmacro\x{18+2}
	\pgfmathsetmacro\y{-7}
	\pgfmathsetmacro\n{8}
	\draw[gray, thick] (\x,\y) -- (\x+\n+1,\y);
	\foreach \i in {1,...,\n} {
		\node (n\i) at (\x+\i, \y) {};
	};
	\foreach \t/\u in {1/6, 6/3, 3/1, 2/5, 5/4, 4/2, 7/8, 8/7} {
		\pgfmathsetmacro\r{(abs(\u-\t))/2}
		\ifthenelse 
		{\t < \u}
		{\draw [thick] (\x+\u,\y) arc (0:90:\r);
			\draw [thick, stealth-] (\x+\t+\r,\y+\r) arc (90:180:\r);
		}
		{\draw [thick] (\x+\u,\y) arc (180:270:\r);
			\draw [thick, stealth-] (\x+\u+\r,\y-\r) arc (270:360:\r);
		}
	}
 	\foreach \i in {1,...,\n} {
		\node[shape=circle,fill=gray, scale=0.5] (n\i) at (\x+\i, \y) {};
	};
  	\foreach \i in {1, 3, 6} {
		\node[shape=circle,fill=black, scale=0.5] (n\i) at (\x+\i, \y) {};
	};
	\node[shape=circle, draw=gray!80, thick, fill=white, scale=0.5] (n0) at (\x+6+1, \y) {};
	//======================================== 
	\node at (27+3, -7) {$= 0$};
	//========================================
	\end{tikzpicture} 
	\]
	\caption{A two-hyper-arc Vassiliev relation; the free leg is depicted as a white disc, while the legs of the second
    hyper arc participating in the relation are shown in black }
	\label{TwoHyperArcVassilievRelation}
\end{figure}



Lemma~\ref{lt1e} below shows that, in the case of more
than one disjoint cycle, one-hyper-arc relations can be
omitted.

 \begin{lemma}\label{lrot}
     Any two hyper arc diagrams
     representing one and the same hyper chord diagram
     are equivalent modulo generalized Vassiliev relations.
 \end{lemma}

 In particular, generalized Vassiliev relations for hyper arc diagrams
 represent generalized Vassiliev relations for hyper chord diagrams.
Figure~\ref{GeneralizedVassilievRelationCD} shows an example of generalized Vassiliev relation for hyper chord diagrams. 


{\bf Proof.} Consider a hyper arc diagram $(\sigma, \alpha)$, $\sigma, \alpha \in \mathbb{S}_m$, where
 the permutation~$\alpha$ contains at least two
disjoint cycles, and pick a leg in~$\alpha$, which will
 be used as a free leg in generalized relations. The alternating sum of $2m-2$ hyper arc diagrams, in which all the hyper arcs but one are the same, while the free leg takes all the $2m-2$ positions close to the legs of $(\sigma, \alpha)$, is equal to zero. Indeed, each hyper arc diagram enters the sum twice, with opposite signs.
 Such an alternating sum contains a single one-hyper-arc Vassiliev relation, which corresponds to the cycle containing the free leg,
 and several two-hyper-arc Vassiliev relations. This proves the
 assertion.\qed

As an immediate corollary, we obtain

 \begin{lemma}\label{lt1e}
 Modulo two-hyper-arc Vassiliev relations, the one-hyper-arc relations are equivalent to
 the requirement of equivalence of hyper-arc-diagrams under cyclic shift.
 \end{lemma}

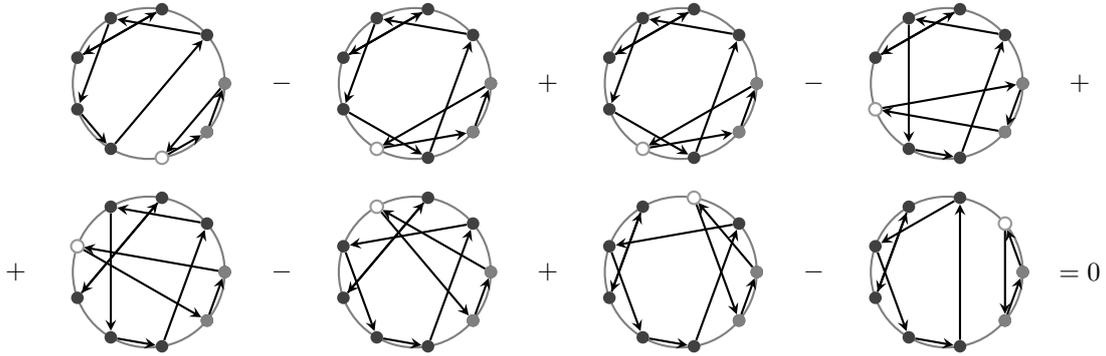
\begin{figure}[ht]
		\[
		\begin{tikzpicture}[baseline={(current bounding box.center)}, scale=0.5]
		//======================================== 1
		\pgfmathsetmacro\x{1}
		\pgfmathsetmacro\y{0}
		\pgfmathsetmacro\rad{2}
		\pgfmathsetmacro\n{9}
		\draw [gray, thick] (0+\x,0+\y) circle (\rad);
		\foreach \i in {1,...,\n} {
			\pgfmathsetmacro\r{\i*(360/\n)}
			\node[shape=circle,fill=darkgray, scale=0.5] (n\i) at ($ (\r:\rad) + (\x,\y) $) {};
		};
		\foreach \t/\u in {1/3, 3/5, 5/6, 6/1, 2/4, 4/2, 8/9, 9/7, 7/8} {
			\draw [thick, -stealth] (n\t) -- (n\u);
		}
        \pgfmathsetmacro\r{7*(360/\n)}
            \node[shape=circle, draw=gray!80, thick, fill=white, scale=0.5] (n0) at ($ (\r:\rad) + (\x,\y) $) {};
            \pgfmathsetmacro\r{8*(360/\n)}
            \node[shape=circle,fill=gray, scale=0.5] (n0) at ($ (\r:\rad) + (\x,\y) $) {};
            \pgfmathsetmacro\r{9*(360/\n)}
            \node[shape=circle,fill=gray, scale=0.5] (n0) at ($ (\r:\rad) + (\x,\y) $) {};
		//======================================== 2
		\node at (4.5, 0) {$-$};
		//========================================
		\pgfmathsetmacro\x{8}
		\pgfmathsetmacro\y{0}
		\pgfmathsetmacro\rad{2}
		\pgfmathsetmacro\n{9}
		\draw [gray, thick] (0+\x,0+\y) circle (\rad);
		\foreach \i in {1,...,\n} {
			\pgfmathsetmacro\r{\i*(360/\n)}
			\node[shape=circle,fill=darkgray, scale=0.5] (n\i) at ($ (\r:\rad) + (\x,\y) $) {};
		};
		\foreach \t/\u in {1/3, 3/5, 5/7, 7/1, 2/4, 4/2, 6/8, 8/9, 9/6} {
			\draw [thick, -stealth] (n\t) -- (n\u);
		}
          \pgfmathsetmacro\r{6*(360/\n)}
            \node[shape=circle, draw=gray!80, thick, fill=white, scale=0.5] (n0) at ($ (\r:\rad) + (\x,\y) $) {};
            \pgfmathsetmacro\r{8*(360/\n)}
            \node[shape=circle,fill=gray, scale=0.5] (n0) at ($ (\r:\rad) + (\x,\y) $) {};
            \pgfmathsetmacro\r{9*(360/\n)}
            \node[shape=circle,fill=gray, scale=0.5] (n0) at ($ (\r:\rad) + (\x,\y) $) {};
		//======================================== 3
		\node at (11.5, 0) {$+$};
		//========================================
		\pgfmathsetmacro\x{15}
		\pgfmathsetmacro\y{0}
		\pgfmathsetmacro\rad{2}
		\pgfmathsetmacro\n{9}
		\draw [gray, thick] (0+\x,0+\y) circle (\rad);
		\foreach \i in {1,...,\n} {
			\pgfmathsetmacro\r{\i*(360/\n)}
			\node[shape=circle,fill=darkgray, scale=0.5] (n\i) at ($ (\r:\rad) + (\x,\y) $) {};
		};
		\foreach \t/\u in {1/3, 3/5, 5/7, 7/1, 2/4, 4/2, 6/8, 8/9, 9/6} {
			\draw [thick, -stealth] (n\t) -- (n\u);
		}
            \pgfmathsetmacro\r{6*(360/\n)}
            \node[shape=circle, draw=gray!80, thick, fill=white, scale=0.5] (n0) at ($ (\r:\rad) + (\x,\y) $) {};
            \pgfmathsetmacro\r{8*(360/\n)}
            \node[shape=circle,fill=gray, scale=0.5] (n0) at ($ (\r:\rad) + (\x,\y) $) {};
            \pgfmathsetmacro\r{9*(360/\n)}
            \node[shape=circle,fill=gray, scale=0.5] (n0) at ($ (\r:\rad) + (\x,\y) $) {};
		//======================================== 4
		\node at (18.5, 0) {$-$};
		//========================================
		\pgfmathsetmacro\x{22}
		\pgfmathsetmacro\y{0}
		\pgfmathsetmacro\rad{2}
		\pgfmathsetmacro\n{9}
		\draw [gray, thick] (0+\x,0+\y) circle (\rad);
		\foreach \i in {1,...,\n} {
			\pgfmathsetmacro\r{\i*(360/\n)}
			\node[shape=circle,fill=darkgray, scale=0.5] (n\i) at ($ (\r:\rad) + (\x,\y) $) {};
		};
		\foreach \t/\u in {1/3, 3/6, 6/7, 7/1, 2/4, 4/2, 5/9, 9/8, 8/5} {
			\draw [thick, -stealth] (n\t) -- (n\u);
		}
            \pgfmathsetmacro\r{5*(360/\n)}
            \node[shape=circle, draw=gray!80, thick, fill=white, scale=0.5] (n0) at ($ (\r:\rad) + (\x,\y) $) {};
            \pgfmathsetmacro\r{8*(360/\n)}
            \node[shape=circle,fill=gray, scale=0.5] (n0) at ($ (\r:\rad) + (\x,\y) $) {};
            \pgfmathsetmacro\r{9*(360/\n)}
            \node[shape=circle,fill=gray, scale=0.5] (n0) at ($ (\r:\rad) + (\x,\y) $) {};
		//========================================
		\node at (25.5, \y) {$+$};
		//======================================== 5
		\pgfmathsetmacro\x{1}
		\pgfmathsetmacro\y{-5}
		\pgfmathsetmacro\rad{2}
		\pgfmathsetmacro\n{9}
		\draw [gray, thick] (0+\x,0+\y) circle (\rad);
		\foreach \i in {1,...,\n} {
			\pgfmathsetmacro\r{\i*(360/\n)}
			\node[shape=circle,fill=darkgray, scale=0.5] (n\i) at ($ (\r:\rad) + (\x,\y) $) {};
		};
		\foreach \t/\u in {1/3, 3/6, 6/7, 7/1, 2/5, 5/2, 4/8, 8/9, 9/4} {
			\draw [thick, -stealth] (n\t) -- (n\u);
		}
        \pgfmathsetmacro\r{4*(360/\n)}
            \node[shape=circle, draw=gray!80, thick, fill=white, scale=0.5] (n0) at ($ (\r:\rad) + (\x,\y) $) {};
            \pgfmathsetmacro\r{8*(360/\n)}
            \node[shape=circle,fill=gray, scale=0.5] (n0) at ($ (\r:\rad) + (\x,\y) $) {};
            \pgfmathsetmacro\r{9*(360/\n)}
            \node[shape=circle,fill=gray, scale=0.5] (n0) at ($ (\r:\rad) + (\x,\y) $) {};
		//======================================== 6
		\node at (-2.5, \y) {$+$};
		\node at (4.5, \y) {$-$};
		//========================================
		\pgfmathsetmacro\x{8}
		\pgfmathsetmacro\y{-5}
		\pgfmathsetmacro\rad{2}
		\pgfmathsetmacro\n{9}
		\draw [gray, thick] (0+\x,0+\y) circle (\rad);
		\foreach \i in {1,...,\n} {
			\pgfmathsetmacro\r{\i*(360/\n)}
			\node[shape=circle,fill=darkgray, scale=0.5] (n\i) at ($ (\r:\rad) + (\x,\y) $) {};
		};
		\foreach \t/\u in {1/4, 4/6, 6/7, 7/1, 2/5, 5/2, 3/8, 8/9, 9/3} {
			\draw [thick, -stealth] (n\t) -- (n\u);
		}
          \pgfmathsetmacro\r{3*(360/\n)}
            \node[shape=circle, draw=gray!80, thick, fill=white, scale=0.5] (n0) at ($ (\r:\rad) + (\x,\y) $) {};
            \pgfmathsetmacro\r{8*(360/\n)}
            \node[shape=circle,fill=gray, scale=0.5] (n0) at ($ (\r:\rad) + (\x,\y) $) {};
            \pgfmathsetmacro\r{9*(360/\n)}
            \node[shape=circle,fill=gray, scale=0.5] (n0) at ($ (\r:\rad) + (\x,\y) $) {};
		//======================================== 7
		\node at (11.5, \y) {$+$};
		//========================================
		\pgfmathsetmacro\x{15}
		\pgfmathsetmacro\y{-5}
		\pgfmathsetmacro\rad{2}
		\pgfmathsetmacro\n{9}
		\draw [gray, thick] (0+\x,0+\y) circle (\rad);
		\foreach \i in {1,...,\n} {
			\pgfmathsetmacro\r{\i*(360/\n)}
			\node[shape=circle,fill=darkgray, scale=0.5] (n\i) at ($ (\r:\rad) + (\x,\y) $) {};
		};
		\foreach \t/\u in {1/4, 4/6, 6/7, 7/1, 3/5, 5/3, 2/8, 8/9, 9/2} {
			\draw [thick, -stealth] (n\t) -- (n\u);
		}
            \pgfmathsetmacro\r{2*(360/\n)}
            \node[shape=circle, draw=gray!80, thick, fill=white, scale=0.5] (n0) at ($ (\r:\rad) + (\x,\y) $) {};
            \pgfmathsetmacro\r{8*(360/\n)}
            \node[shape=circle,fill=gray, scale=0.5] (n0) at ($ (\r:\rad) + (\x,\y) $) {};
            \pgfmathsetmacro\r{9*(360/\n)}
            \node[shape=circle,fill=gray, scale=0.5] (n0) at ($ (\r:\rad) + (\x,\y) $) {};
		//======================================== 8
		\node at (18.5, \y) {$-$};
		//========================================
		\pgfmathsetmacro\x{22}
		\pgfmathsetmacro\y{-5}
		\pgfmathsetmacro\rad{2}
		\pgfmathsetmacro\n{9}
		\draw [gray, thick] (0+\x,0+\y) circle (\rad);
		\foreach \i in {1,...,\n} {
			\pgfmathsetmacro\r{\i*(360/\n)}
			\node[shape=circle,fill=darkgray, scale=0.5] (n\i) at ($ (\r:\rad) + (\x,\y) $) {};
		};
		\foreach \t/\u in {1/8, 8/9, 9/1, 2/4, 4/6, 6/7, 7/2, 3/5, 5/3} {
			\draw [thick, -stealth] (n\t) -- (n\u);
		}
            \pgfmathsetmacro\r{1*(360/\n)}
            \node[shape=circle, draw=gray!80, thick, fill=white, scale=0.5] (n0) at ($ (\r:\rad) + (\x,\y) $) {};
            \pgfmathsetmacro\r{8*(360/\n)}
            \node[shape=circle,fill=gray, scale=0.5] (n0) at ($ (\r:\rad) + (\x,\y) $) {};
            \pgfmathsetmacro\r{9*(360/\n)}
            \node[shape=circle,fill=gray, scale=0.5] (n0) at ($ (\r:\rad) + (\x,\y) $) {};
		\node at (25.5, \y) {$= 0$};
		\end{tikzpicture} 
		\]
		\caption{A two-hyper-chord generalized Vassiliev relation for hyper chord diagrams }
		\label{GeneralizedVassilievRelationCD}
	\end{figure}

\newpage
 
A function on hyper chord diagrams satisfying generalized Vassiliev relations
is called a \emph{generalized weight system}. The genus of a hyper chord diagram
is an example of a generalized weight system. Indeed, the $2\ell$ summands
in a generalized Vassiliev relation split into $\ell$ pairs of opposite signs,
such that both diagrams in each pair have the same genus, since one of
them is obtained from the other by sliding the free leg of the hyperchord~$a$
along the hyper chord~$b$.

\section{Lie algebras weight systems on permutations}\label{s4}

To each metrized Lie algebra, one associates a weight system. 
However, until recently, for most Lie algebras, no efficient ways to compute the value
of this weight system for a specific Lie algebra on a given chord diagram
were known. The situation changed in 2022, when the recurrence
relations for computing the values of the $\gl(N)$-weight systems, for arbitrary
$N=1,2,\dots$ were deduced~\cite{KL22,ZY22}. Introduction of these recurrence
relations required extending the $\gl(N)$-weight systems to permutations.
Later, this construction has been extended to other classical series of Lie
algebras~\cite{KY24}.  Below, we show that these weight
systems on permutations satisfy the generalized Vassiliev
relations we introduce in the present paper.

\subsection{$\gl$-weight system for  {hyper arc diagrams} as a generalized weight system}

We have already given an example of a generalized weight system,
namely, that of the genus of a hyper arc diagram. 
Another example, and a very important one, is given
by the $\gl$-weight system, see~\cite{KL22,ZY22}. It takes values in the ring
$\BC[N,C_1,C_2,\dots]$ of polynomials in infinitely many variables.
This weight system on {hyper arc diagrams} can be defined 
in purely combinatorial terms by the following recurrence relations:
\begin{itemize}
    \item $w_{\gl}$ is multiplicative with respect to the connected sum (concatenation)
    of  {hyper arc diagrams}. This implies, in particular, that for the empty hyper {arc} diagram the value of $w_{\gl}$ is~$1$.
    \item For {the cyclic hyper arc diagram $(\sigma, \alpha)$, where $\alpha=\sigma$ is} the standard cyclic permutation of length~$m$ (where the cyclic order on the set of permuted elements
    is consistent with that induced by the premutation), the value of $w_\gl$ is~$C_m$.
    \item  For an arbitrary {hyper arc diagram $(\sigma, \alpha)$,} {$\alpha$} $\in \BS_m$,
    and for any two neighboring elements $k$, $k+1$ out of $\{1,2,\dots,m\}$, the invariant $w_{\gl}$
    satisfies the equation shown in Fig.~\ref{f-glrr}.
\begin{figure}[ht]
\begin{multline*}
    w_{\gl}
        \left(\begin{tikzpicture}[baseline={([yshift=-.5ex]current bounding box.center)}]
		\pgfmathsetmacro\x{1}
		\pgfmathsetmacro\r{1}
		\draw[->,thick] (-1.5,0) --  (1.5,0);
            \node[below left] at (-\x, 0) {$k$};
            \node[below right] at (\x, 0) {$k+1$};
		\draw [thick] (\x,0) arc (30:90:2*\r);
		\draw [thick, -stealth reversed] (\x,0) arc (30:60:2*\r);
		\draw [thick] (-\x, 0) arc (150:90:2*\r);
		\draw [thick, -stealth] (-\x,0) arc (150:120:2*\r);
		\draw [thick] (\x,0) arc (-30:-90:2*\r);
		\draw [thick, -stealth] (\x,0) arc (-30:-60:2*\r);
		\draw [thick] (-\x, 0) arc (210:270:2*\r);
		\draw [thick, -stealth  reversed] (-\x, 0) arc (210:240:2*\r);
            \node[shape=circle,fill=darkgray, scale=0.5] at (-\x, 0) {};
            \node[shape=circle,fill=darkgray, scale=0.5] at (\x, 0) {};
		\pgfmathsetmacro\t{0.9}
		\pgfmathsetmacro\e{0.2}
		\draw (-\x*\t,\r*\t+\e) node {$a$};
		\draw (\x*\t,\r*\t+\e) node {$c$};
		\draw (-\x*\t,-\r*\t-\e) node {$b$};
		\draw (\x*\t,-\r*\t-\e) node {$d$};
	\end{tikzpicture}\right) 
        - w_{\gl}\left(\begin{tikzpicture}[baseline={([yshift=-.5ex]current bounding box.center)}]
		\pgfmathsetmacro\x{0.6}
		\pgfmathsetmacro\r{1}
		\draw[->,thick] (-1.5,0) --  (1.5,0);
            \node[below left] at (-\x*1.1, 0) {$k$};
            \node[below right] at (\x*1.1, 0) {$k+1$};
		\draw [thick] (-\x,0) arc (0:55:\r);
		\draw [thick, -stealth reversed] (-\x,0) arc (0:30:\r);
		\draw [thick] (\x, 0) arc (180:125:\r);
		\draw [thick, -stealth] (\x,0) arc (180:150:\r);
		\draw [thick] (-\x,0) arc (0:-55:\r);
		\draw [thick, -stealth] (-\x,0) arc (0:-30:\r);
		\draw [thick] (\x, 0) arc (180:235:\r);
		\draw [thick, -stealth  reversed] (\x, 0) arc (180:210:\r);
            \node[shape=circle,fill=darkgray, scale=0.5] at (-\x, 0) {};
            \node[shape=circle,fill=darkgray, scale=0.5] at (\x, 0) {};
		\pgfmathsetmacro\t{0.9}
		\draw (-\x/\t,\r*\t*\t) node {$a$};
		\draw (\x/\t,\r*\t*\t) node {$c$};
		\draw (-\x/\t,-\r*\t*\t) node {$b$};
		\draw (\x/\t,-\r*\t*\t) node {$d$};
	\end{tikzpicture}\right) = \\
        \\
        = w_{\gl}\left(\begin{tikzpicture}[baseline={([yshift=-.5ex]current bounding box.center)}]
		\pgfmathsetmacro\x{1}
		\pgfmathsetmacro\r{1}
		\draw[->,thick] (-1.5,0) --  (1.5,0);
            \node[above] at (0, 0) {$k$};
		\pgfmathsetmacro\a{65}
		\draw [thick] (0,0) arc (0:-\a:\r);
		\draw [thick, -stealth] (0,0) arc (0:-\a*0.6:\r);
		\draw [thick] (0, 0) arc (180:180+\a:\r);
		\draw [thick, -stealth  reversed] (0, 0) arc (180:180+\a*0.6:\r);
		\pgfmathsetmacro\b{55}
		\draw [thick] (0,1) arc (90:90-\b:\r);
		\draw [thick, stealth-] (0,1) arc (90:90+\b:\r);
            \node[shape=circle,fill=darkgray, scale=0.5] at (0, 0) {};
		\pgfmathsetmacro\t{0.9}
		\draw (-\x*\t,\r*\t) node {$a$};
		\draw (\x*\t,\r*\t) node {$c$};
		\draw (-\x*\t,-\r*\t) node {$b$};
		\draw (\x*\t,-\r*\t) node {$d$};
	\end{tikzpicture}\right)-
		w_{\gl}\left(\begin{tikzpicture}[baseline={([yshift=-.5ex]current bounding box.center)}]
		\pgfmathsetmacro\x{1}
		\pgfmathsetmacro\r{1}
		\draw[->,thick] (-1.5,0) --  (1.5,0);
            \node[below] at (0, 0) {$k$};
		\pgfmathsetmacro\a{65}
		\draw [thick] (0,0) arc (180:180-\a:\r);
		\draw [thick, -stealth] (0,0) arc (180:180-\a*0.6:\r);
		\draw [thick] (0, 0) arc (0:\a:\r);
		\draw [thick, -stealth  reversed] (0, 0) arc (0:\a*0.6:\r);
		\pgfmathsetmacro\b{55}
		\draw [thick] (0,-1) arc (270:270-\b:\r);
		\draw [thick, stealth-] (0,-1) arc (270:270+\b:\r);
            \node[shape=circle, fill=darkgray, scale=0.5] at (0, 0) {};
		\pgfmathsetmacro\t{0.9}
		\draw (-\x*\t,\r*\t) node {$a$};
		\draw (\x*\t,\r*\t) node {$c$};
		\draw (-\x*\t,-\r*\t) node {$b$};
		\draw (\x*\t,-\r*\t) node {$d$};
	\end{tikzpicture}\right)
\end{multline*}
\caption{The recurrence relation for the universal $\gl$-weight system }\label{f-glrr}
    
\end{figure}

The  {hyper arc diagrams} in the left hand side depict two neighboring vertices and  arcs that are incident to them.
In the  {hyper arc diagrams} on the right, these two vertices are replaced with a single one.
All the other vertices and  {legs} are the same for all the four  {hyper arc diagrams} in the relation. An example of such equation is depicted in Fig.~\ref{f-glrre}.  

\begin{figure} [ht]
\begin{multline*}
    w_{\gl}
        \left(\begin{tikzpicture}[baseline={([yshift=-.5ex]current bounding box.center)}, scale=0.5]
		\pgfmathsetmacro\x{0}
		\pgfmathsetmacro\y{0}
		\pgfmathsetmacro\n{6}
  		\draw[gray, thick] (\x,\y) -- (\x+\n+1,\y);
		\foreach \i in {1,...,\n} {
			\node (n\i) at (\x+\i, \y) {};
		};
		\foreach \t/\u in {1/5, 5/2, 2/1, 3/6, 6/4, 4/3} {
			\pgfmathsetmacro\r{(abs(\u-\t))/2}
			\ifthenelse 
			{\t < \u}
			{\draw [gray, thick] (\x+\u,\y) arc (0:90:\r);
				\draw [gray, thick, stealth-] (\x+\t+\r,\y+\r) arc (90:180:\r);
			}
			{\draw [gray, thick] (\x+\u,\y) arc (180:270:\r);
				\draw [gray, thick, stealth-] (\x+\u+\r,\y-\r) arc (270:360:\r);
			}
		}
		\foreach \t/\u in {1/5, 5/2, 6/4, 4/3} {
			\pgfmathsetmacro\r{(abs(\u-\t))/2}
			\ifthenelse 
			{\t < \u}
			{\draw [thick] (\x+\u,\y) arc (0:90:\r);
				\draw [thick, stealth-] (\x+\t+\r,\y+\r) arc (90:180:\r);
			}
			{\draw [thick] (\x+\u,\y) arc (180:270:\r);
				\draw [thick, stealth-] (\x+\u+\r,\y-\r) arc (270:360:\r);
			}
		}
		\foreach \i in {1,...,\n} {
			\node[shape=circle,fill=darkgray, scale=0.5] (n\i) at (\x+\i, \y) {};
		};
		\foreach \i in {4, 5} {
			\node[shape=circle, draw=black!90, thick, fill=white, scale=0.5] (n0) at (\x+\i, \y) {};
		};
	\end{tikzpicture}\right) 
        - w_{\gl}        \left(\begin{tikzpicture}[baseline={([yshift=-.5ex]current bounding box.center)}, scale=0.5]
		\pgfmathsetmacro\x{0}
		\pgfmathsetmacro\y{0}
		\pgfmathsetmacro\n{6}
  		\draw[gray, thick] (\x,\y) -- (\x+\n+1,\y);
		\foreach \i in {1,...,\n} {
			\node (n\i) at (\x+\i, \y) {};
		};
		\foreach \t/\u in {1/5, 5/2, 2/1, 3/6, 6/4, 4/3} {
			\pgfmathsetmacro\r{(abs(\u-\t))/2}
			\ifthenelse 
			{\t < \u}
			{\draw [gray, thick] (\x+\u,\y) arc (0:90:\r);
				\draw [gray, thick, stealth-] (\x+\t+\r,\y+\r) arc (90:180:\r);
			}
			{\draw [gray, thick] (\x+\u,\y) arc (180:270:\r);
				\draw [gray, thick, stealth-] (\x+\u+\r,\y-\r) arc (270:360:\r);
			}
		}
		\foreach \t/\u in {1/5, 5/2, 6/4, 4/3} {
			\pgfmathsetmacro\r{(abs(\u-\t))/2}
			\ifthenelse 
			{\t < \u}
			{\draw [thick] (\x+\u,\y) arc (0:90:\r);
				\draw [thick, stealth-] (\x+\t+\r,\y+\r) arc (90:180:\r);
			}
			{\draw [thick] (\x+\u,\y) arc (180:270:\r);
				\draw [thick, stealth-] (\x+\u+\r,\y-\r) arc (270:360:\r);
			}
		}
		\foreach \i in {1,...,\n} {
			\node[shape=circle,fill=darkgray, scale=0.5] (n\i) at (\x+\i, \y) {};
		};
		\foreach \i in {4, 5} {
			\node[shape=circle, draw=black!90, thick, fill=white, scale=0.5] (n0) at (\x+\i, \y) {};
		};
	\end{tikzpicture}\right)  = \\
        \\
        = w_{\gl}        \left(\begin{tikzpicture}[baseline={([yshift=-.5ex]current bounding box.center)}, scale=0.5]
		\pgfmathsetmacro\x{0}
		\pgfmathsetmacro\y{0}
		\pgfmathsetmacro\n{5}
  		\draw[gray, thick] (\x,\y) -- (\x+\n+1,\y);
		\foreach \i in {1,...,\n} {
			\node (n\i) at (\x+\i, \y) {};
		};
		\foreach \t/\u in {1/3, 3/5, 5/4, 4/2, 2/1} {
			\pgfmathsetmacro\r{(abs(\u-\t))/2}
			\ifthenelse 
			{\t < \u}
			{\draw [gray, thick] (\x+\u,\y) arc (0:90:\r);
				\draw [gray, thick, stealth-] (\x+\t+\r,\y+\r) arc (90:180:\r);
			}
			{\draw [gray, thick] (\x+\u,\y) arc (180:270:\r);
				\draw [gray, thick, stealth-] (\x+\u+\r,\y-\r) arc (270:360:\r);
			}
		}
		\foreach \t/\u in {5/4, 4/2, 1/3} {
			\pgfmathsetmacro\r{(abs(\u-\t))/2}
			\ifthenelse 
			{\t < \u}
			{\draw [thick] (\x+\u,\y) arc (0:90:\r);
				\draw [thick, stealth-] (\x+\t+\r,\y+\r) arc (90:180:\r);
			}
			{\draw [thick] (\x+\u,\y) arc (180:270:\r);
				\draw [thick, stealth-] (\x+\u+\r,\y-\r) arc (270:360:\r);
			}
		}
		\foreach \i in {1,...,\n} {
			\node[shape=circle,fill=darkgray, scale=0.5] (n\i) at (\x+\i, \y) {};
		};
		\foreach \i in {4} {
			\node[shape=circle, draw=black!90, thick, fill=white, scale=0.5] (n0) at (\x+\i, \y) {};
		};
	\end{tikzpicture}\right)
        - w_{\gl}        \left(\begin{tikzpicture}[baseline={([yshift=-.5ex]current bounding box.center)}, scale=0.5]
		\pgfmathsetmacro\x{0}
		\pgfmathsetmacro\y{0}
		\pgfmathsetmacro\n{5}
  		\draw[gray, thick] (\x,\y) -- (\x+\n+1,\y);
		\foreach \i in {1,...,\n} {
			\node (n\i) at (\x+\i, \y) {};
		};
		\foreach \t/\u in {1/4, 4/3, 3/5, 5/2, 2/1} {
			\pgfmathsetmacro\r{(abs(\u-\t))/2}
			\ifthenelse 
			{\t < \u}
			{\draw [gray, thick] (\x+\u,\y) arc (0:90:\r);
				\draw [gray, thick, stealth-] (\x+\t+\r,\y+\r) arc (90:180:\r);
			}
			{\draw [gray, thick] (\x+\u,\y) arc (180:270:\r);
				\draw [gray, thick, stealth-] (\x+\u+\r,\y-\r) arc (270:360:\r);
			}
		}
		\foreach \t/\u in {1/4, 4/3, 5/2} {
			\pgfmathsetmacro\r{(abs(\u-\t))/2}
			\ifthenelse 
			{\t < \u}
			{\draw [thick] (\x+\u,\y) arc (0:90:\r);
				\draw [thick, stealth-] (\x+\t+\r,\y+\r) arc (90:180:\r);
			}
			{\draw [thick] (\x+\u,\y) arc (180:270:\r);
				\draw [thick, stealth-] (\x+\u+\r,\y-\r) arc (270:360:\r);
			}
		}
		\foreach \i in {1,...,\n} {
			\node[shape=circle,fill=darkgray, scale=0.5] (n\i) at (\x+\i, \y) {};
		};
		\foreach \i in {4} {
			\node[shape=circle, draw=black!90, thick, fill=white, scale=0.5] (n0) at (\x+\i, \y) {};
		};
	\end{tikzpicture}\right) 
\end{multline*}
  \caption{An example of applying the recurrence relation
  shown in Fig.~\ref{f-glrr} }\label{f-glrre}  
\end{figure}

In the exceptional case $\sigma(k+1)=k$, the relation acquires the form shown in Fig.~\ref{f-glrrx}.
\begin{figure}
\begin{multline*}
    w_{\gl}
        \left(\begin{tikzpicture}[baseline={([yshift=-.5ex]current bounding box.center)}]
		\pgfmathsetmacro\x{0.6}
		\pgfmathsetmacro\r{2}
		\pgfmathsetmacro\a{80}
		\draw[->,thick] (-1.5,0) --  (1.5,0);
            \node[below left] at (-\x, 0) {$k$};
            \node[below right] at (\x, 0) {$k+1$};
		\draw [thick] (\x,0) arc (30:\a:\r);
		\draw [thick, -stealth reversed] (\x,0) arc (30:\a*0.6:\r);
		\draw [thick] (-\x, 0) arc (150:180-\a:\r);
		\draw [thick, -stealth] (-\x,0) arc (150:180-\a*0.6:\r);
		\draw [thick] (\x,0) arc (0:-180:\x);
		\draw [thick, -stealth] (\x,0) arc (0:-90:\x);
            \node[shape=circle,fill=darkgray, scale=0.5] at (-\x, 0) {};
            \node[shape=circle,fill=darkgray, scale=0.5] at (\x, 0) {};
		\pgfmathsetmacro\t{0.5}
		\pgfmathsetmacro\e{0.15}
		\draw (-3*\x*\t,\r*\t+\e) node {$a$};
		\draw (3*\x*\t,\r*\t+\e) node {$b$};
	\end{tikzpicture}\right) 
        - w_{\gl}\left(\begin{tikzpicture}[baseline={([yshift=-.5ex]current bounding box.center)}]
		\pgfmathsetmacro\x{0.6}
		\pgfmathsetmacro\r{1}
		\draw[->,thick] (-1.5,0) --  (1.5,0);
            \node[below] at (-\x*1.1, 0) {$k$};
            \node[below] at (\x*1.1, 0) {$k+1$};
		\draw [thick] (-\x,0) arc (0:55:\r);
		\draw [thick, -stealth reversed] (-\x,0) arc (0:30:\r);
		\draw [thick] (\x, 0) arc (180:125:\r);
		\draw [thick, -stealth] (\x,0) arc (180:150:\r);
		\draw [thick] (-\x,0) arc (180:0:\x);
		\draw [thick, -stealth] (-\x,0) arc (180:90:\x);
            \node[shape=circle,fill=darkgray, scale=0.5] at (-\x, 0) {};
            \node[shape=circle,fill=darkgray, scale=0.5] at (\x, 0) {};
		\pgfmathsetmacro\t{0.9}
		\draw (-2.2*\x*\t,\r*\t) node {$a$};
		\draw (2.2*\x*\t,\r*\t) node {$b$};
	\end{tikzpicture}\right) = \\
        \\
        = C_1\times w_{\gl}\left(\begin{tikzpicture}[baseline={([yshift=-.5ex]current bounding box.center)}]
		\pgfmathsetmacro\x{1}
		\pgfmathsetmacro\r{1}
		\draw[->,thick] (-1.1,0) --  (1.1,0);
		\pgfmathsetmacro\b{55}
		\draw [thick, stealth-] (0,1) arc (90:90-\b:\r);
		\draw [thick, stealth-] (0,1) arc (90:90+\b:\r);
		\pgfmathsetmacro\t{0.9}
		\draw (-1.2*\x*\t,\r*\t) node {$a$};
		\draw (1.2*\x*\t,\r*\t) node {$b$};
	\end{tikzpicture}\right)-
		N\times w_{\gl}\left(\begin{tikzpicture}[baseline={([yshift=-.5ex]current bounding box.center)}]
		\pgfmathsetmacro\x{1}
		\pgfmathsetmacro\r{1}
		\draw[->,thick] (-1.5,0) --  (1.5,0);
            \node[below] at (0, 0) {$k$};
		\pgfmathsetmacro\a{65}
		\draw [thick] (0,0) arc (180:180-\a:\r);
		\draw [thick, -stealth] (0,0) arc (180:180-\a*0.6:\r);
		\draw [thick] (0, 0) arc (0:\a:\r);
		\draw [thick, -stealth  reversed] (0, 0) arc (0:\a*0.6:\r);
            \node[shape=circle, fill=darkgray, scale=0.5] at (0, 0) {};
		\pgfmathsetmacro\t{0.9}
		\draw (-\x*\t,\r*\t) node {$a$};
		\draw (\x*\t,\r*\t) node {$b$};
	\end{tikzpicture}\right)
\end{multline*}
\caption{The form the recurrence relation acquires for $\alpha(k+1)=k$}\label{f-glrrx}
\end{figure}

An example of applying this recurrence relation is shown 
in Fig.~\ref{f-glrry}.

\begin{figure}
\begin{multline*}
    w_{\gl}
        \left(\begin{tikzpicture}[baseline={([yshift=-.5ex]current bounding box.center)}, scale=0.5]
		\pgfmathsetmacro\x{0}
		\pgfmathsetmacro\y{0}
		\pgfmathsetmacro\n{3}
  		\draw[gray, thick] (\x,\y) -- (\x+\n+1,\y);
		\foreach \i in {1,...,\n} {
			\node (n\i) at (\x+\i, \y) {};
		};
		\foreach \t/\u in {1/3, 3/2, 2/1} {
			\pgfmathsetmacro\r{(abs(\u-\t))/2}
			\ifthenelse 
			{\t < \u}
			{\draw [thick] (\x+\u,\y) arc (0:90:\r);
				\draw [thick, stealth-] (\x+\t+\r,\y+\r) arc (90:180:\r);
			}
			{\draw [thick] (\x+\u,\y) arc (180:270:\r);
				\draw [thick, stealth-] (\x+\u+\r,\y-\r) arc (270:360:\r);
			}
		}
		\foreach \i in {1,...,\n} {
			\node[shape=circle,fill=darkgray, scale=0.5] (n\i) at (\x+\i, \y) {};
		};
		\foreach \i in {2, 3} {
			\node[shape=circle, draw=black!90, thick, fill=white, scale=0.5] (n0) at (\x+\i, \y) {};
		};
	\end{tikzpicture}\right) 
        - w_{\gl}        \left(\begin{tikzpicture}[baseline={([yshift=-.5ex]current bounding box.center)}, scale=0.5]
		\pgfmathsetmacro\x{0}
		\pgfmathsetmacro\y{0}
		\pgfmathsetmacro\n{3}
  		\draw[gray, thick] (\x,\y) -- (\x+\n+1,\y);
		\foreach \i in {1,...,\n} {
			\node (n\i) at (\x+\i, \y) {};
		};
		\foreach \t/\u in {1/2, 2/3, 3/1} {
			\pgfmathsetmacro\r{(abs(\u-\t))/2}
			\ifthenelse 
			{\t < \u}
			{\draw [thick] (\x+\u,\y) arc (0:90:\r);
				\draw [thick, stealth-] (\x+\t+\r,\y+\r) arc (90:180:\r);
			}
			{\draw [thick] (\x+\u,\y) arc (180:270:\r);
				\draw [thick, stealth-] (\x+\u+\r,\y-\r) arc (270:360:\r);
			}
		}
		\foreach \i in {1,...,\n} {
			\node[shape=circle,fill=darkgray, scale=0.5] (n\i) at (\x+\i, \y) {};
		};
		\foreach \i in {2, 3} {
			\node[shape=circle, draw=black!90, thick, fill=white, scale=0.5] (n0) at (\x+\i, \y) {};
		};
	\end{tikzpicture}\right) 
        = C_1w_{\gl}        \left(\begin{tikzpicture}[baseline={([yshift=-.5ex]current bounding box.center)}, scale=0.5]
		\pgfmathsetmacro\x{0}
		\pgfmathsetmacro\y{0}
		\pgfmathsetmacro\n{1}
  		\draw[gray, thick] (\x,\y) -- (\x+\n+1,\y);
		\foreach \i in {1,...,\n} {
			\node (n\i) at (\x+\i, \y) {};
		};
		\foreach \i in {1,...,\n} {
			\node[shape=circle,fill=darkgray, scale=0.5] (n\i) at (\x+\i, \y) {};
		};
	\end{tikzpicture}\right)
        - N w_{\gl}        \left(\begin{tikzpicture}[baseline={([yshift=-.5ex]current bounding box.center)}, scale=0.5]
		\pgfmathsetmacro\x{0}
		\pgfmathsetmacro\y{0}
		\pgfmathsetmacro\n{2}
  		\draw[gray, thick] (\x,\y) -- (\x+\n+1,\y);
		\foreach \i in {1,...,\n} {
			\node (n\i) at (\x+\i, \y) {};
		};
		\foreach \t/\u in {1/2, 2/1} {
			\pgfmathsetmacro\r{(abs(\u-\t))/2}
			\ifthenelse 
			{\t < \u}
			{\draw [thick] (\x+\u,\y) arc (0:90:\r);
				\draw [thick, stealth-] (\x+\t+\r,\y+\r) arc (90:180:\r);
			}
			{\draw [thick] (\x+\u,\y) arc (180:270:\r);
				\draw [thick, stealth-] (\x+\u+\r,\y-\r) arc (270:360:\r);
			}
		}
		\foreach \i in {1,...,\n} {
			\node[shape=circle,fill=darkgray, scale=0.5] (n\i) at (\x+\i, \y) {};
		};
		\foreach \i in {2} {
			\node[shape=circle, draw=black!90, thick, fill=white, scale=0.5] (n0) at (\x+\i, \y) {};
		};
	\end{tikzpicture}\right) 
\end{multline*}
\caption{An example of applying the recurrence relation 
for $\alpha(k+1)=k$}\label{f-glrry}
\end{figure}

\end{itemize}

For each specific value of~$N$, $N=2,3,4,\dots$, this weight system coincides with
the $\gl(N)$-weight system defined as follows.

Let~$m$ be a positive integer and let $\BS_m$ be the permutation group of the elements
$\{1,2,\dots,m\}$; for an arbitrary $\alpha\in \BS_m$, let us set
$$
w_{\gl(N)}(\alpha)=\sum_{i_1,\dots,i_m=1}^NE_{i_1i_{\alpha(1)}}E_{i_2i_{\alpha(2)}}\dots E_{i_mi_{\alpha(m)}}\in U(\gl(N)).
$$

The constructed element belongs to the center $ZU(\gl(N))$~\cite{KL22,ZY23}.
In addition, it is invariant under conjugation by the standard cyclic permutation:
$$
w_{\gl(N)}(\sigma^{-1}\alpha\sigma)=\sum_{i_1,\dots,i_m=1}^N E_{i_2i_{\sigma(2)}}\dots E_{i_mi_{\sigma(m)}}E_{i_1i_{\sigma(1)}}=w_{\gl(N)}(\alpha).
$$
In particular, the Casimir element $C_m$ corresponds to the cyclic permutation
$(1,2,\dots,m-1,m)$. On the other hand, each chord diagram with~$n$ arcs
can be considered as an involution without fixed points on the set of $m=2n$
elements. The value of $w_{\gl(N)}$ on the permutation determined by this involution
coincides with the value of the $\gl(N)$-weight system on the corresponding chord diagram.
For example, for the chord diagram~$K_n$, which consists of~$n$ chords each of which
intersects one another, we have $w_{\gl(N)}(K_n)=w_{\gl(N)}((1~n{+}1)(2~n{+}2)\dots(n~2n))$.

\begin{theorem}\label{t-glgV}
The $\gl$-weight system on hyper chord diagrams satisfies the generalized
Vassiliev relations.
\end{theorem}

{\bf Proof.} 

A standard proof of the fact that a mapping taking a chord diagram
to the universal enveloping algebra of a metrized Lie algebra
satisfies the four-term relations and is, therefore, a weight system
is based on the properties of structure coefficients of the Lie algebra,
see, e.g.~\cite{CDBook12} or Chapter~6 in~\cite{LZ03}.
This proof can be reproduced for the weight system $w_\gl$ on permutations
and the generalized Vassiliev relations as well. It is easier, however,
to make use of the recurrence relation for~$w_\gl$ shown in Fig.~\ref{f-glrr}.

Let us prove the two-hyper-arc generalized Vassiliev relations for $w_\gl$.
Pick two disjoint cycles $v_1,v_2\in V(\alpha)$ 
in a permutation~$\alpha\in\BS_m$, a free leg in~$v_1$, and write down the
two-hyper-arc generalized Vassiliev relation for this choice.  
The left-hand side of this relation is an alternating sum of values
of $w_\gl$ on the permutations obtained from~$\alpha$ by putting 
the free leg at all the $2\ell$ positions neighboring the $\ell=\ell(v_2)$
legs of the cycle~$v_2$. By means of the recurrence relation in Fig.~\ref{f-glrr},
the difference of values of $w_\gl$ on the two permutations, for which the
free leg is on the left and on the right of the same leg of~$v_2$,
is equal to the difference of its values on two permutations in~$\BS_{m-1}$
obtained by gluing together the cycles $v_1,v_2$ in~$\alpha$ in a single 
cycle in two different ways, both of which respect orientation of the cycle.
In this way we replace the alternating sum of values of $w_\gl$ on $2\ell$
permutations in $\BS_m$ with the alternating sum of its values on $2\ell$
permutations in $\BS_{m-1}$.

Now the latter sum splits into $\ell$ pairs of values of $w_\gl$ on
coinciding permutations, having the opposite signs. Indeed, for a leg~$k$
of~$v_2$, the first permutation on the right in Fig.~\ref{f-glrr}
coincides with the second permutation on the right for the leg $v_2(k)$,
and we are done.

For  the one-hyper-arc  type generalized Vassiliev relations, the proof is similar.

\begin{example}
    For the one-leg element shown in Fig.~\ref{OneHyperArcVassilievRelation},
the value of the $w_\gl$-weight system is
\begin{eqnarray*}
    w_\gl((1,4)(2,5,3))-w_\gl((1,4)(2,5,3))+w_\gl((1,4)(2,5,3))-w_\gl((1,4)(2,5,3))\\
    =(C_2N^2-(C_1^2+C_2^2+C_3)N+(C_1^2+C_1+C_3)C_2)
-(-C_3N+(C_1+C_3)C_2)\\+ (-C_3N+(C_1+C_3)C_2)-   
(C_2N^2-(C_1^2+C_2^2+C_3)N+(C_1^2+C_1+C_3)C_2)\\
=0.
\end{eqnarray*}
\end{example}

\begin{example}
    For the two-leg element shown in Fig.~\ref{TwoHyperArcVassilievRelation},
    introduce notation
\begin{eqnarray*}
  \alpha_1=(1, 8) (2, 7, 4) (3, 6, 5),
  \alpha_2=(1, 7, 4) (2, 8) (3, 6, 5),
  \alpha_3=(1, 7,4) (2, 6, 5) (3, 8),\\
  \alpha_4=(1, 7,3) (2, 6,5) (4,8),
   \alpha_5=(1, 7,3) (2, 5, 4) (6,8),
  \alpha_6=(1, 6,3) (2, 5, 4) (7,8)
 \end{eqnarray*}
for the permutations entering it.
The value of the $w_\gl$-weight system on these permutations is
\begin{eqnarray*}
w_\gl(\alpha_1)&=&-C_2^2N^3+C_2(C_1^2+C_2^2+2C_3)N^2\\
&&-C_2(2C_1C_2+2C_1^2C_2+2C_2C_3+C_4)N+C_2(C_1^4+C_2^2+2C_1^2C_3+C_3^2)\\
w_\gl(\alpha_2)&=&C_2N^4-(C_1^2+2C_2^2+2C_3)N^3+(3C_1C_2+3C_1^2C_2+C_2^3+4C_2C_3+C_4)N^2\\    
&&-(C_1^3+C_1^4+C_2^2+3C_1C_2^2+2C_1^2C_2^2+C_1C_3+2C_1^2C_3+2C_2^2C_3+C_3^2+C_2C_4)N\\
&&+C_2(C_1^2+C_1^3+C_1^4+C_2^2+C_1C_3+2C_1^2C_3+C_3^2)\\
w_\gl(\alpha_3)&=&2C_2N^4-(2C_1^2+3C_2^2+3C_3)N^3+C_2(1+C_1+5C_1^2+C_2^2+6C_3)N^2\\ 
&&-(C_1^2-2C_1^3+2C_1^4+4C_2^2+4C_1C_2^2+2C_1^2C_2^2+C_3-6C_1C_3+4C_1^2C_3+2C_2^2C_3+2C_3^2+C_2C_4-C_5)\\
&&+C_1C_2-2C_1^2C_2+2C_1^3C_2+C_1^4C_2+C_2^3+2C_2C_3+2C_1C_2C_3+2C_1^2C_2C_3+C_2C_3^2-3C_1C_4\\
w_\gl(\alpha_4)&=&2C_2N^4-(2C_1^2+3C_2^2+3C_3)N^3+(1+C_1+5C_1^2+C_2^2+6C_3)C_2N^2\\
&&-(C_1^2-2C_1^3+2C_1^4+4C_2^2+4C_1C_2^2+2C_1^2C_2^2+C_3-6C_1C_3
+4C_1^2C_3+2C_2^2C_3+2C_3^2+C_2C_4-C_5)N\\
&&+C_1C_2-2C_1^2C_2+2C_1^3C_2+C_1^4C_2+C_2^3+2C_2C_3+2C_1C_2C_3+2C_1^2C_2C_3+C_2^2C_3-3C_1C_4\\    
w_\gl(\alpha_5)&=&C_2N^4-(C_1^2+2C_2^2+2C_3)N^3+(3C_1C_2+3C_1^2C_2+C_2^3+4C_2C_3+C_4)N^2\\
&&-(C_1^3+C_1^4+C_2^2+3C_1C_2^2+2C_1^2C_2^2+C_1C_3+2C_1^2C_3+2C_2^2C_3+C_3^2+C_2C_4)N\\
&&+(C_1^2+C_1^3+C_1^4+C_2^2+C_1C_3+2C_1^2C_3+C_3^2)C_2\\
w_\gl(\alpha_6)&=& -C_2^2N^3+(C_1^2+C_2^2+2C_3)C_2N^2-\\ &&(2C_1C_2+2C_1^2C_2+2C_2C_3+C_4)C_2N+(C_1^4+C_2^2+2C_1^2C_3+C_3^2)C_2,   
\end{eqnarray*}
so that 
$$
w_\gl(\alpha_1)-w_\gl(\alpha_2)+w_\gl(\alpha_3)-w_\gl(\alpha_4)+w_\gl(\alpha_5)-w_\gl(\alpha_6)=0.   
$$

\end{example}

\subsection{Hyper map topology of permutations and standard substitution}\label{ss42}

Define the \emph{number of faces} $f(\alpha)$
of a permutation~$\alpha$ as follows.
A permutation considered up to a cyclic shift
can be treated as a hypermap with a single vertex,
which is the disc having the Wilson loop
as the boundary.
Namely, each cycle in the decomposition of a permutation
into the product of disjoint cycles is 
a hyperedge, which can be represented in the following
way. Replace the $\ell$-edge with a $2\ell$-gon
whose odd-numbered edges one-to-one correspond
to the permuted points, while even-numbered edges
one-to-one correspond to the arrows in the graph of
the permutation. Then, at each leg of the $\ell$-edge,
move the end of the entering arrow slightly to the {right} and
attach the $2\ell$-gon
to the Wilson loop along the $\ell$ odd-numbered edges
to the segments obtained in this way, see Fig.~\ref{NumberOfFacesExample132} and ~\ref{NumberOfFacesExample123}.

\begin{figure}[ht]
	\[
	\begin{tikzpicture}[baseline={(current bounding box.center)}, scale=0.5]

	//======================================== 1
	\pgfmathsetmacro\x{1}
	\pgfmathsetmacro\y{0}
	\pgfmathsetmacro\rad{3}
	\pgfmathsetmacro\n{3}
	\draw [gray, thick] (0+\x,0+\y) circle (\rad);
	\foreach \i in {1,...,\n} {
		\pgfmathsetmacro\r{\i*(360/\n)-30}
		\node at ($ (\r:\rad+0.6) + (\x,\y) $) {\i};
		\node[shape=circle,fill=gray, scale=0.5] (n\i) at ($ (\r:\rad) + (\x,\y) $) {};
	};
	\foreach \t/\u in {1/2, 2/3, 3/1} {
		\draw [ultra thick, stealth-] (n\t) -- (n\u);
	}
	//========================================
        \draw[gray, ultra thick, -latex] (5.5,0) -- (6.5,0);
	//======================================== 2
	\pgfmathsetmacro\x{11}
	\pgfmathsetmacro\y{0}
	\pgfmathsetmacro\rad{3}
	\pgfmathsetmacro\e{15}
        \pgfmathsetmacro\n{3}
	\pgfmathsetmacro\d{6} 
	\draw [gray] (0+\x,0+\y) circle (\rad);
        \foreach \i in {1,...,\d} {
		\pgfmathsetmacro\r{\i*(360/\d)}
            \ifthenelse{\isodd{\i}}
            {\draw[ultra thick, gray] ($ (\r+\e:\rad) + (\x,\y) $) arc (\r+\e:\r+360/\d-\e:\rad);
            \draw[ultra thick, dash pattern={on 1pt off 3pt}, white] ($ (\r+\e:\rad) + (\x,\y) $) arc (\r+\e:\r+360/\d-\e:\rad);
            }
	};
 	\foreach \i in {1,...,\d} {
		\pgfmathsetmacro\r{\i*(360/\d)}
            \ifthenelse{\isodd{\i}}
            {\node[shape=circle,fill=gray, scale=0.5] (n\i) at ($ (\r+\e:\rad) + (\x,\y) $) {};}
            {\node[shape=circle,fill=gray, scale=0.5] (n\i) at ($ (\r-\e:\rad) + (\x,\y) $) {};}
	};
	\foreach \t/\u in {2/5, 6/3, 4/1} {
		\draw [ultra thick, -stealth, shorten <= 0.85 cm] (n\t) -- (n\u);
		\draw [ultra thick, shorten >= 2.2 cm] (n\t) -- (n\u);
        };
        \foreach \i in {1,...,\n} {
		\pgfmathsetmacro\r{\i*(360/\n)-30}
		\node at ($ (\r:\rad+0.6) + (\x,\y) $) {\i};
	};
        \pgfmathsetmacro\minrad{\rad/2}
  	\foreach \i in {1,...,\n} {
		\pgfmathsetmacro\r{\i*(360/\n)};
            \node (k\i) at ($ (\r+90:\minrad) + (\x,\y) $) {};
	};
        \filldraw[gray, draw=none, pattern=north east lines, opacity=0.3] (k1.center) -- (k2.center) -- (k3.center) -- cycle;
        \filldraw[gray, draw=none, pattern=north west lines, opacity=0.3] (n2.center) -- (n1.center) -- (k3.center) -- cycle;
        \filldraw[gray, draw=none, pattern=north east lines, opacity=0.3] (n3.center) -- (n4.center) -- (k1.center) -- cycle;
        \filldraw[gray, draw=none, pattern=north east lines, opacity=0.3] (n5.center) -- (n6.center) -- (k2.center) -- cycle;
	//========================================
        \draw[gray, ultra thick, latex-] (15.5,0) -- (16.5,0);
	//======================================== 3
	\pgfmathsetmacro\x{21}
	\pgfmathsetmacro\y{0}
	\pgfmathsetmacro\rad{3}
	\pgfmathsetmacro\e{15}
        \pgfmathsetmacro\n{3}
	\pgfmathsetmacro\d{6} 
 	\foreach \i in {1,...,\d} {
		\pgfmathsetmacro\r{\i*(360/\d)}
            \ifthenelse{\isodd{\i}}
            {\node[shape=circle,fill=gray, scale=0.5] (n\i) at ($ (\r+\e:\rad) + (\x,\y) $) {};}
            {\node[shape=circle,fill=gray, scale=0.5] (n\i) at ($ (\r-\e:\rad) + (\x,\y) $) {};}
	};
	\foreach \t/\u in {1/6, 5/4, 3/2} {
		\draw [ultra thick, -stealth] (n\t) -- (n\u);
        };
	\foreach \t/\u in {6/5, 4/3, 2/1} {
		\draw [ultra thick, gray] (n\t) -- (n\u);
		\draw [ultra thick, white, dash pattern={on 1pt off 3pt}] (n\t) -- (n\u);
        };
        \foreach \i in {1,...,\n} {
		\pgfmathsetmacro\r{\i*(360/\n)-30}
		\node at ($ (\r:\rad+0.6) + (\x,\y) $) {\i};
	};
        \filldraw[gray, draw=none, pattern=north east lines, opacity=0.3] (n1.center) -- (n2.center) -- (n3.center) -- (n4.center) -- (n5.center) -- (n6.center) -- cycle;
	//========================================
 
	\end{tikzpicture} 
	\]
	\caption{In the center, the hypermap for the permutation $\alpha = (132)=\sigma^{ -1}$, $m=3$ is shown}
	\label{NumberOfFacesExample132}
\end{figure}
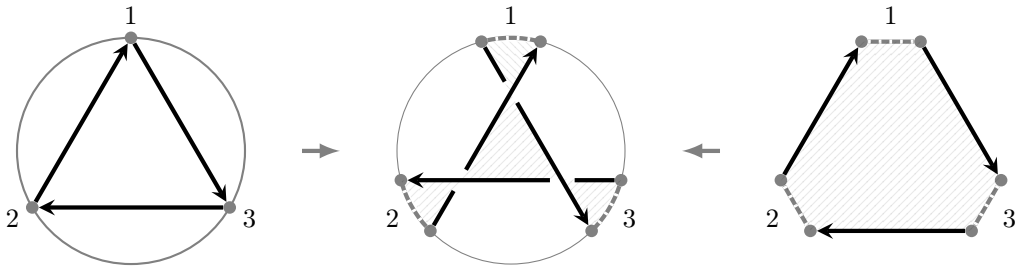
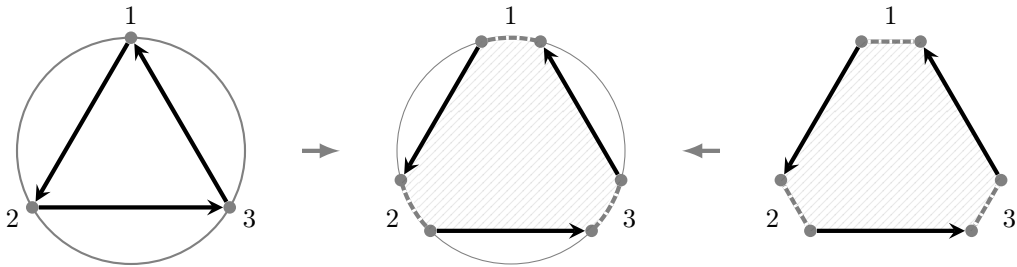
\begin{figure}[ht]
	\[
	\begin{tikzpicture}[baseline={(current bounding box.center)}, scale=0.5]

	//======================================== 1
	\pgfmathsetmacro\x{1}
	\pgfmathsetmacro\y{0}
	\pgfmathsetmacro\rad{3}
	\pgfmathsetmacro\n{3}
	\draw [gray, thick] (0+\x,0+\y) circle (\rad);
	\foreach \i in {1,...,\n} {
		\pgfmathsetmacro\r{\i*(360/\n)-30}
		\node at ($ (\r:\rad+0.6) + (\x,\y) $) {\i};
		\node[shape=circle,fill=gray, scale=0.5] (n\i) at ($ (\r:\rad) + (\x,\y) $) {};
	};
	\foreach \t/\u in {1/2, 2/3, 3/1} {
		\draw [ultra thick, -stealth] (n\t) -- (n\u);
	}
	//========================================
        \draw[gray, ultra thick, -latex] (5.5,0) -- (6.5,0);
	//======================================== 2
	\pgfmathsetmacro\x{11}
	\pgfmathsetmacro\y{0}
	\pgfmathsetmacro\rad{3}
	\pgfmathsetmacro\e{15}
        \pgfmathsetmacro\n{3}
	\pgfmathsetmacro\d{6} 
	\draw [gray] (0+\x,0+\y) circle (\rad);
        \foreach \i in {1,...,\d} {
		\pgfmathsetmacro\r{\i*(360/\d)}
            \ifthenelse{\isodd{\i}}
            {\draw[ultra thick, gray] ($ (\r+\e:\rad) + (\x,\y) $) arc (\r+\e:\r+360/\d-\e:\rad);
            \draw[ultra thick, dash pattern={on 1pt off 3pt}, white] ($ (\r+\e:\rad) + (\x,\y) $) arc (\r+\e:\r+360/\d-\e:\rad);
            }
	};
 	\foreach \i in {1,...,\d} {
		\pgfmathsetmacro\r{\i*(360/\d)}
            \ifthenelse{\isodd{\i}}
            {\node[shape=circle,fill=gray, scale=0.5] (n\i) at ($ (\r+\e:\rad) + (\x,\y) $) {};}
            {\node[shape=circle,fill=gray, scale=0.5] (n\i) at ($ (\r-\e:\rad) + (\x,\y) $) {};}
	};
	\foreach \t/\u in {2/3, 4/5, 6/1} {
		\draw [ultra thick, -stealth] (n\t) -- (n\u);
        };
        \foreach \i in {1,...,\n} {
		\pgfmathsetmacro\r{\i*(360/\n)-30}
		\node at ($ (\r:\rad+0.6) + (\x,\y) $) {\i};
	};
        \filldraw[gray, draw=none, pattern=north east lines, opacity=0.3] (n1.center) -- (n2.center) -- (n3.center) -- (n4.center) -- (n5.center) -- (n6.center) -- cycle;
	//========================================
        \draw[gray, ultra thick, latex-] (15.5,0) -- (16.5,0);
	//======================================== 3
	\pgfmathsetmacro\x{21}
	\pgfmathsetmacro\y{0}
	\pgfmathsetmacro\rad{3}
	\pgfmathsetmacro\e{15}
        \pgfmathsetmacro\n{3}
	\pgfmathsetmacro\d{6} 
 	\foreach \i in {1,...,\d} {
		\pgfmathsetmacro\r{\i*(360/\d)}
            \ifthenelse{\isodd{\i}}
            {\node[shape=circle,fill=gray, scale=0.5] (n\i) at ($ (\r+\e:\rad) + (\x,\y) $) {};}
            {\node[shape=circle,fill=gray, scale=0.5] (n\i) at ($ (\r-\e:\rad) + (\x,\y) $) {};}
	};
	\foreach \t/\u in {6/1, 4/5, 2/3} {
		\draw [ultra thick, -stealth] (n\t) -- (n\u);
        };
	\foreach \t/\u in {6/5, 4/3, 2/1} {
		\draw [ultra thick, gray] (n\t) -- (n\u);
		\draw [ultra thick, white, dash pattern={on 1pt off 3pt}] (n\t) -- (n\u);
        };
        \foreach \i in {1,...,\n} {
		\pgfmathsetmacro\r{\i*(360/\n)-30}
		\node at ($ (\r:\rad+0.6) + (\x,\y) $) {\i};
	};
        \filldraw[gray, draw=none, pattern=north east lines, opacity=0.3] (n1.center) -- (n2.center) -- (n3.center) -- (n4.center) -- (n5.center) -- (n6.center) -- cycle;
	//========================================
	\end{tikzpicture} 
	\]
	\caption{In the center, the hypermap for the permutation $\alpha = (123)=\sigma$, $m=3$ is shown}
	\label{NumberOfFacesExample123}
\end{figure}

\newpage

After doing this with each disjoint cycle of a permutation~$\alpha\in\BS_m$, we obtain a $2$-dimensional surface 
with boundary, which possesses usual topological 
characteristics: number of connected components of
the boundary (faces) $f(\alpha)$, genus
$g(\alpha)=m-f(\alpha)+2$ (to which we will refere below as the \emph{genus of the
hyper chord diagram}~$\alpha$), and so on. 
It is easy to see that the number $f(\alpha)$
coincides with the number of disjoint cycles $c(\sigma^{-1}\alpha)$ in the product
permutation $\sigma^{-1}\alpha$.
Note that the resulting surface is necessarily orientable.
For an involution without fixed points,
the notions of the number of faces and the genus
coincide with the standard ones for the corresponding chord diagram. 

\begin{remark}
    The notion of hypermap is essentially equivalent to
    the notion of bicolored map, that is, a map whose
    vertices are colored in two colors so that no two
    vertices of the same color are connected by an edge.
    Indeed, color all the vertices of a hypermap with the first color, and replace each $\ell$-hyperedge 
    with a vertex of the second color and $\ell$ edges
    attached to it. 
    The construction in the opposite direction is straightforward.
\end{remark}

Call the substitution $C_k\mapsto N^{k-1}$, $k=1,2,\dots$,
the \emph{standard representation} substitution for Casimirs.
For ordinary chord diagrams, the following assertion is well-known, see~\cite{BN95}.

\begin{theorem}\cite{KL25}
    The value $w_\gl(\alpha)$ on a permutation~$\alpha\in \BS_m$
    under the standard representation substitution for Casimirs
    becomes $N^{f(\alpha)-1}=N^{m-g(\alpha)+1}$, where $f(\alpha)$ 
    is the number of faces, and $g(\alpha)$ is the genus of the permutation $\alpha$.
\end{theorem}

{\bf Proof.} In the standard representation of $\gl(N)$,
the element $E_{ij}$ is taken to the matrix unit, that is,
to the $N\times N$-matrix having~$0$ everywhere
except for the entry on the intersection of the~$i$~th row
and $j$~th column, where the entry is~$1$. 
Given a permutation~$\alpha\in S_m$, each summand
$$
{\rm Tr}\ {\rm St}(E_{i_1i_{\alpha(1)}})\dots {\rm St}(E_{i_mi_{\alpha(m)}})
$$
in the definition of~$w_{\gl(N)}(\alpha)$ 
in the standard representation is either~$1$
or~$0$. It is~$1$ if and only if the following chain
of equalities holds:
$$
i_{\alpha(1)}=i_2,\quad i_{\alpha(2)}=i_3,\dots,
i_{\alpha(m)}=i_1,
$$
that is, iff all the indices 
on each connected component of the boundary of the
hypermap of~$\alpha$ are the same.
Since the index between~$1$ and~$N$ on each
boundary component can be chosen independently,
this completes the proof. \qed

\subsection{Harish-Chandra isomorphism and Schur substitution}

The Casimir elements $C_k$, $k=1,2,\dots$, are one
possible sequence of generators of the centers 
$ZU\gl(N)$ of universal enveloping algebras. 
Other convenient choices are shifted power sums~\cite{OO96} and
one-part Schur polynomials $S_k$. The precise definition of the latter
depends on the normalization chosen, 
and we will follow the one described below in this section.

The algebra $U(\gl(N))$ admits the decomposition into the direct sum
\begin{equation}\label{HCd}
	U(\gl(N))=(\fn_-U(\gl(N))+U(\gl(N))\fn_+)\oplus U(\fh),
\end{equation}
where $\fn_-$, $\fh$ and $\fn_+$ are, respectively, the Lie  subalgebras of low-triangular, diagonal, and upper-triangular matrices in $\gl(N)$.

\begin{definition}    
The \emph{Harish-Chandra projection} is the linear projection to the second summand in the right hand side of~\eqref{HCd},
$$
\phi:U(\gl(N))\to U(\fh)=\BC[E_{11},\cdots,E_{NN}].
$$
\end{definition}

\begin{theorem}[Harish-Chandra isomorphism, \cite{O91}]\label{t-HCi}
The restriction of the Harish-Chandra projection to the center $ZU(\gl(N))$ is an injective homomorphism of
commutative algebras, and its isomorphic image consists of
symmetric functions in the shifted diagonal matrix units $x_i=E_{ii}+\frac{N+1}2-i$, $i=1,\dots,N$.
\end{theorem}

The isomorphism of the theorem allows one to identify the codomain
$\BC[C_1,C_2,\dots,C_N]$ of the weight system $w_{\gl(N)}$
with the ring of symmetric functions in the generators
$x_1,\dots,x_N$.
The expression for Casimir elements
under this isomorphism is described by the following Perelomov-Popov formula~\cite{PP68}:
\begin{equation*}
1-N\,u-\sum_{k=1}^\infty\phi(C_k)\,u^{k+1}=
\prod_{i=1}^N\frac{1-(x_i+\frac{N+1}2)\,u}{1-(x_i-\frac{N+1}2)u}.
\end{equation*}
Another convenient set of generators is represented by complete symmetric functions, or,
which is the same, \emph{one-part Schur polynomials} $S_k$, $k=1,2,\dots$, defined by the following power series expansion: 
\begin{equation*}
1+\sum_{k=1}^\infty S_k u^k=\prod_{i=1}^N\frac1{1-x_iu}.
\end{equation*}
The resulting relation between the generators $C_k$ and $S_k$
is encoded in the following equality of
generating series:
$$
1-Nu-\sum_{k=1}^\infty\phi(C_k)u^{k+1}=\left(\frac{1-(N+1)u/2)}{1-(N-1)u/2}\right)^N
\frac{1+\sum_{i=1}^\infty S_i\cdot \left(\frac{u}{1-(N-1)u/2}\right)^{i}}{1+\sum_{i=1}^\infty S_i\cdot\left(\frac{u}{1-(N+1)u/2}\right)^{i}}.
$$

We call the corresponding substitution
\begin{eqnarray*}
    C_1&=&S_1,\\
    C_2&=&2S_2-S_1^2-\frac1{12}(N-1)N(N+1),\\
    C_3&=&3S_3-3S_1S_2+S_1^3+NS_2-\frac12(N+1)S_1^2
    -\frac1{4}(N-1)(N+1)S_1-\frac1{24}(N-1)N^2(N+1),\\
    \dots&=&\dots
\end{eqnarray*}
the \emph{Schur substitution}. It is polynomial and invertible, so that 
each variable~$S_m$, $m=1,2,\dots$ can be expressed as a polynomial in 
$C_1,\dots,C_m$ whose coefficients are polynomials in~$N$.

\subsection{$\gl$-weight system for the inverse permutation}

The values of $w_\gl$ on two mutually inverse permutations, when expressed as polynomials
in the Casimir elements, look unrelated to
one another. The situation changes dramatically,
however, if we make the Schur substitution.

\begin{theorem}\label{tiS}
    Let~$\alpha\in\BS_m$ be a permutation, and let
    $\alpha^{-1}$ be its inverse. Then the
    Schur substitution for $w_\gl(\alpha^{-1})$ 
   can be obtained from that for $w_\gl(\alpha)$
   by  substituting $S_i\mapsto (-1)^iS_i$,
   $i=1,2,3,\dots$ and further multiplying by $(-1)^m$.    
\end{theorem}

{\bf Proof.} Consider the automorphism of the Lie algebra~$\gl(N)$
which takes the matrix unit $E_{i,j}$ to $-E_{N+1-j,N+1-i}$, $i,j=1,\dots,N$
(reflection of an $N\times N$-matrix with respect to the secondary diagonal,
composed with sign changing).
This automorphism extends to an automorphism of the universal enveloping
algebra $U\gl(N)$ and, in particular, of its center $ZU\gl(N)$.
The latter takes the value $w_{\gl(N)}(\alpha)$ on a permutation~$\alpha$
to $w_{\gl(N)}(\alpha^{-1})$, its value on the inverse permutation.

The
automorphism
of
$U\gl(N)$
in
question
preserves
the
Harish-Chandra
decomposition~(\ref{HCd}). Under this isomorphism, the generator $x_i$ is mapped to $x_{\bar i}= -x_{N+1-i}$  which proves the assertion. 
\qed

\begin{example} For the permutation $\alpha=(123)$, we have
\begin{eqnarray*}
w_\gl(\alpha)&=&C_3\\
&=&3S_3-3S_1S_2+S_1^3+NS_2-\frac12(N+1)S_1^2
    -\frac1{4}(N-1)(N+1)S_1-\frac1{24}(N-1)N^2(N+1),
\end{eqnarray*}
while
\begin{eqnarray*}
w_\gl(\alpha^{-1})&=&C_3-NC_2+C_1^2\\
&=&3S_3-3S_1S_2+S_1^3-NS_2+\frac12(N+1)S_1^2
    -\frac1{4}(N-1)(N+1)S_1+\frac1{24}(N-1)N^2(N+1)\\
    &=&(-1)^3w_\gl(\alpha)|_{S_1\mapsto -S_1,S_2\mapsto S_2,S_3\mapsto-S_3}.
\end{eqnarray*}

\end{example}

\begin{corollary}
    If  $\alpha$ is an involution  without  fixed points
    (which corresponds to  a chord diagram), 
    so that $\alpha=\alpha^{-1}$,
    then $w_\gl(\alpha)$, when expressed in terms of 
    one-part Schur polynomials,
    does not contain monomials with odd numbers of odd-indexed 
    variables~$S_k$.
\end{corollary}
\subsection{Averaging $\gl$-weight system}

Let $A_m=\frac1{m!}\sum_{\alpha\in \BS_m}w_\gl(\alpha)$ be the average value of $w_\gl$ on permutations of~$m$ elements; it is a 
polynomial in $\BC[N][C_1,C_2,C_3,\dots]$. Below, for a nonnegative integer~$\ell$,
we will use notation $(a)_\ell$ for the falling factorial,
$$
(a)_\ell=a(a-1)(a-2)\dots(a-\ell+1).
$$

\begin{theorem}
    In the standard representation, the sum of the values of~$w_\gl$
    on permutations of~$m$ elements is $(N+m-1)_{m-1}$. In other words,
    the number of permutations of~$m$ elements whose hypermap has $k$
    boundary components is the coefficient of $N^{k-1}$ in the polynomial
    $(N+1)(N+2)\dots(N+m-1)$, that is, the Stirling number of the first kind
    $|s(m,k)|$.
\end{theorem}

This assertion follows immediately from Jucys' theorem~\cite{J71},
which states that in the generating function
$$
\prod_{k=2}^m(t+J_k),
$$
where $J_k=\sum_{i=1}^{k-1}(i,k)\in \BC[\BS_m]$, $k=2,3,\dots$ are the \emph{Jucys--Murphy elements},
the coefficient of~$t^\ell$ is the sum of all permutations in~$S_m$ having~$\ell$
cycles. Indeed, for a hypermap defined by a pair of permutations $(\sigma,\alpha)$, the number
$f(\alpha)$ of connected components of the 
boundary coincides with the number of cycles
in the product permutation $\sigma^{-1}\alpha$.
In turn, since multiplication by $\sigma^{-1}$
is a one-to-one mapping of $\BS_m$ into itself,
the number of permutations in $\BS_m$ with~$\ell$
boundary components coincides with the number of those
with~$\ell$ disjoint cycles.

We conjectured the following assertion basing on the result of computer experiments.
Recently, it has been proved by M.~Zaitsev, see Appendix.

\begin{theorem}\label{taS}
Under the Schur substitution for Casimirs, the polynomial~$A_m$ 
becomes a linear combination of the form
$$ S_m-a_2(N)(N+m-1)_2S_{m-2}+a_4(N)(N+m-1)_4S_{m-4}-a_6(N)(N+m-1)_6S_{m-6}+\dots.$$
Here $a_{2k}$ are polynomials of degree~$k$ in~$N$.
\end{theorem}

Explicitly, the leading terms in  the linear combination  of  the theorem are
\begin{eqnarray*}
&&S_m-\frac1{24}(N-1)(N+m-1)_2S_{m-2}+
\frac1{240}(N-1)(5N-3)(N+m-1)_4S_{m-4}\\
&&-
\frac1{4032}(N-1)(35N^2-28N+9)(N+m-1)_6S_{m-6}+\dots.
\end{eqnarray*}

Since any linear combination of one-part Schur polynomials $S_1,S_2,\dots$
is a $\tau$-function for the KP hierarchy~\cite{KP70,S83}, we immediately deduce
 
\begin{corollary}
    Any generating function of the form
    $1+\sum_{m=1}^\infty c_mA_m u^m$ is a
    one-parameter family of tau-functions for the 
    KP integrable hierarchy of partial differential equations.
\end{corollary}

Therefore, the mean value of the universal $\gl$-weight system on permutations
yields one more instance of a combinatorial solution of the KP hierarchy, see also~\cite{CKL20,KL15}.

The following assertion specifies the exact form of the polynomials~$a_k$.
It is also proved in the Appendix written by M.~Zaitsev.

\begin{theorem}
The exponential generating function $\cA$ 
for the polynomials~$a_k$ is
\begin{eqnarray*}
\cA(v)&=&1+\sum_{k=1}^\infty a_k(N)\frac{v^k}{k!}\\
&=&1+\frac1{24}(N-1)\frac{v^2}{2!}
+\frac1{240}(N-1)(5N-3)\frac{v^4}{4!}+
\frac1{4032}(N-1)(35N^2-28N+9)\frac{v^6}{6!}+\dots\\
&=&\left(\frac{e^{v/2}-e^{-v/2}}{v}\right)^{N-1}.
\end{eqnarray*}
\end{theorem}

Equivalently, the polynomial~$A_m$ can be obtained 
in the following way.
Apply
to the monomial~$x^{N+m-1}$ the linear operator $x^{-N+1}\cS(d/dx)^{N-1}$, where
$$
\cS(v)=\frac{e^{v/2}-e^{-v/2}}v,
$$
eliminate negative powers of~$x$ in the resulting power series, 
and replace each monomial~$x^k$ by~$S_k$.

\begin{remark}
    The averaging formula in Theorem~\ref{taS}
    may be considered as a one-parameter generalization
    of the Bernoulli (or Faulhaber's) formula for power sums.
    The latter reads
    $$
    \sum_{k=1}^nk^m=
    \frac{n^{m+1}}{m+1}+\frac12n^m+
    \sum_{\ell=2}^m(m)_{\ell-1}\frac{B_\ell}{\ell!}
    n^{m-\ell+1},
    $$
where $(m)_\ell$, as above, denotes the $\ell$-th
falling factorial
and~$B_\ell$ are the Bernoulli numbers,
$$
\sum_{\ell=0}^\infty B_\ell\frac{t^\ell}{\ell!}=\frac{t}{1-e^{-t}}.
$$

\end{remark}

\begin{remark}
    For a given~$m$,  the averaging  formula can also be treated  as
    the  action of an $(m+1)\times(m+1)$ lower  triangular matrix
    on the  vector $(S_0=1,S_1,\dots,S_m)$.
    Similarly to the case of Bernoulli numbers,
    the  inverse of this matrix, which expresses
    the one-part Schur polynomials as certain  linear combinations of the 
    averages of $w_\sl$,
    has a very simple 
    form and also possesses nice properties.
    
\end{remark}

\subsection{$so$-weight system}

Following~\cite{KY24}, we define in this section a multiplicative weight system $w_\so$ taking values in the ring of polynomials in the generators $N$, and $C_2,C_4,C_6,\dots$ labeled by even 
nonnegative integers. Similarly to the case of the $w_\fgl$ weight system, the definition of $w_\so$ 
proceeds by extending it to the set of permutations (on any number of permuted elements). This extension is defined by a set of recurrence relations that are close to those for the $w_\fgl$ case. The defining relations presented below are motivated by the requirement that for the Lie algebra $\so(N)$ (with $N$ of any parity) the specialization of $w_\so$ taking $C_k$ to the corresponding  Casimir element is given by
 \begin{equation}\label{eq:wsonsigma}
	w_{\so(N)}(\alpha)=\sum_{i_1,\cdots,i_m=1}^N F_{i_1i_{\alpha(1)}}F_{i_2i_{\alpha(2)}}\cdots F_{i_mi_{\alpha(m)}}\in U\so(N),
\end{equation}
where $F_{ij}$ are the \emph{standard generators} of $\so(N)$:
$$
F_{ij}=E_{ij}-E_{\bar j\bar i},\qquad \bar i=N+1-i.
$$

 
Before formulating defining relations for $w_\so$, we introduce some notation used in these relations. For $w_\fgl$, the defining relations
are naturally expressed in terms of digraphs of permutations.
Such a digraph just shows where each permuted element is being taken.

The invariant $w_\so$ constructed below possesses an additional symmetry that does not hold for the $\fgl$ case:
\begin{quote}
\emph{assume that the permutation $\alpha'$ is obtained from $\alpha$ by  inverting one of its disjoint cycles. In this case, the values $w_\so(\alpha)$ and $w_\so(\alpha')$ differ by the sign factor $(-1)^\ell$ where $\ell$ is the length of the cycle.}
    \end{quote}

In particular, we have 

\begin{corollary}
	The value of the $\so$-weight system on the inverse permutation $\alpha^{-1}$
	of a permutation~$\alpha\in\BS_m$ is given by
	$w_\so(\alpha^{-1})=(-1)^mw_\so(\alpha)$.
\end{corollary}

This symmetry leads to the following convention. Along with the digraphs of permutations, we consider more general graphs, which we call extended permutation graphs.

\begin{definition} An \emph{extended permutation graph} is a graph with the following properties:
\begin{itemize}
\item the set of vertices of the graph is linearly ordered, which is depicted by placing them on an additional oriented line in the order compatible with the orientation of that line;
\item each vertex has valency~$2$, in particular, the number of edges is equal to the number of vertices;
\item for each half-edge it is specified whether it is a \emph{head} (marked with an arrow) or a \emph{tail}. For two half-edges adjacent to every vertex one of them is a tail and the other is a head. However, we allow for the edges of the graph to have two heads, or two tails, or a head and a tail.
\end{itemize}
\end{definition}

In an ordinary permutation graph, each edge has one tail and one head. 
We extend $w_\so$ to 
extended permutation graphs by the following additional relation: 
a change of the tail and the head for the two
half-edges adjacent to any vertex results in the multiplication 
of the value of the invariant by $-1$:

$$
    w_{\so}
        \left(\begin{tikzpicture}[baseline={([yshift=-.5ex]current bounding box.center)}]
		\pgfmathsetmacro\x{0}
		\pgfmathsetmacro\y{0.5}
		\pgfmathsetmacro\r{1}
		\draw[->,thick] (-1.5,\y) --  (1.5,\y);
		\draw [thick] (\x,\y) arc (15:60:2*\r);
		\draw [thick] (-\x, \y) arc (165:120:2*\r);
		\draw [thick, -stealth reversed] (-\x,\y) arc (165:150:2*\r);
            \node[shape=circle,fill=darkgray, scale=0.5] at (\x, \y) {};

            \node[scale=0.5] at (\x, 0) {};
	\end{tikzpicture}\right) 
        = -w_{\so}\left(\begin{tikzpicture}[baseline={([yshift=-.5ex]current bounding box.center)}]
		\pgfmathsetmacro\x{0}
		\pgfmathsetmacro\y{0.5}
		\pgfmathsetmacro\r{1}
		\draw[->,thick] (-1.5,\y) --  (1.5,\y);
		\draw [thick] (\x,\y) arc (15:60:2*\r);
		\draw [thick, -stealth reversed] (\x,\y) arc (15:30:2*\r);
		\draw [thick] (-\x, \y) arc (165:120:2*\r);
            \node[shape=circle,fill=darkgray, scale=0.5] at (\x, \y) {};

            \node[scale=0.5] at (\x, 0) {};
	\end{tikzpicture}\right)
$$

By applying this transformation several times, every extended permutation graph can be reduced to a usual permutation graph (characterized by the additional property that every edge has one head and one tail). This permutation graph is not unique but the symmetry of $w_\so$ formulated above implies that this ambiguity does not affect the extension of $w_\so$.

\begin{definition}\label{def:so}
The universal polynomial invariant $w_\so$ is the function on the set of permutations of any number of elements (or equivalently, on the set of permutation graphs) taking values in the ring of polynomials in the generators $N$, and $C_2,C_4,\dots$ and defined by the following set of relations (axioms):
\begin{itemize}
\item $w_{\so}$ is multiplicative with respect to the concatenation of the permutation graphs, $w_\so(\alpha_1\#\alpha_2)=w_\so(\alpha_1)w_\so(\alpha_2)$. 
As a corollary, for the empty graph (with no vertices) the value of $w_{\so}$ is equal to $1$.
\item A change of orientation of any cycle of length~$\ell$ in the graph results in multiplication of the value of the invariant~$w_\so$ by~$(-1)^\ell$.
\item For a cyclic permutation of even number of elements {\rm(}with the cyclic order on the set of permuted elements  compatible with the permutation{\rm)} $1\mapsto2\mapsto\cdots\mapsto m\mapsto1$,
 the value of $w_{\so}$ is the standard generator,

$$
    w_{\so}
        \left(\begin{tikzpicture}[baseline={([yshift=-.5ex]current bounding box.center)}]
		\pgfmathsetmacro\x{0}
		\pgfmathsetmacro\y{0.5}
		\pgfmathsetmacro\r{1}
		\pgfmathsetmacro\t{0.5}
		\pgfmathsetmacro\e{0.25}
		\draw[->,thick] (\x-\t,\y) --  (\x+6*\t,\y);            

		\draw [thick] (\x, \y) arc (180:0:\t/2);
		\draw [thick, ->] (\x,\y) arc (180:80:\t/2);
		\draw [thick] (\x+\t, \y) arc (180:0:\t/2);
		\draw [thick, ->] (\x+\t, \y) arc (180:80:\t/2);
		\draw [thick] (\x+4*\t, \y) arc (180:0:\t/2);
		\draw [thick, ->] (\x+4*\t,\y) arc (180:80:\t/2);
		\draw [thick, ->] (\x+2*\t,\y) arc (180:80:\t/2);
		\draw [thick] (\x+4*\t,\y) arc (0:90:\t/2);
		\draw [thick] (\x, \y) arc (210:330:5*\t*0.57735);
		\draw [thick, -<] (\x,\y) arc (210:270:5*\t*0.57735);

            \node[shape=circle,fill=darkgray, scale=0.4] at (\x, \y) {};
            \node at (\x,\y-\e) {$1$};
            \node[shape=circle,fill=darkgray, scale=0.4] at (\x+\t, \y) {};
            \node at (\x+\t,\y-\e) {$2$};
            \node[shape=circle,fill=darkgray, scale=0.4] at (\x+2*\t, \y) {};
            \node at (\x+2*\t,\y-\e) {$3$};
            \node[shape=circle,fill=darkgray, scale=0.4] at (\x+4*\t, \y) {};
            \node[shape=circle,fill=darkgray, scale=0.4] at (\x+5*\t, \y) {};
            \node at (\x+5*\t+0.05,\y-\e) {$m$};

            \node at (\x+3*\t,\y+\e/2) {...};

            \node[scale=0.5] at (\x, 0) {};
	\end{tikzpicture}\right) 
        = C_m,\qquad m\text{ is even}.
$$
 
\item {\rm(}\textbf{Recurrence Rule}{\rm)} For the graph of an arbitrary permutation $\alpha$,
and for any pair of its vertices labelled by consecutive integers $k,k+1$, we have for the values of the $w_\so$ weight system the recurrence relation shown in Fig.~ \ref{f-rec-so1}.

\begin{figure}
\begin{multline*}
    w_\so
        \left(\begin{tikzpicture}[baseline={([yshift=-.5ex]current bounding box.center)}]
		\pgfmathsetmacro\x{0.4}
		\pgfmathsetmacro\y{0.3}
		\pgfmathsetmacro\r{1}
		\draw[->,thick] (-1,0) --  (1,0);
            \node at (-\x, -\y) {$k$};
            \node at (\x, -\y) {$k+1$};
		\draw [thick] (\x,0) arc (30:80:2*\r);
		\draw [thick, -stealth reversed] (\x,0) arc (30:45:2*\r);
		\draw [thick] (\x,0) arc (10:70:2*\r);
		\draw [thick] (-\x, 0) arc (150:100:2*\r);
		\draw [thick] (-\x,0) arc (170:110:2*\r);
		\draw [thick, -stealth reversed] (-\x,0) arc (170:155:2*\r);
            \node[shape=circle,fill=darkgray, scale=0.5] at (-\x, 0) {};
            \node[shape=circle,fill=darkgray, scale=0.5] at (\x, 0) {};
	\end{tikzpicture}\right) 
        =  w_\so\left(\begin{tikzpicture}[baseline={([yshift=-.5ex]current bounding box.center)}]
		\pgfmathsetmacro\x{0.4}
		\pgfmathsetmacro\y{0.3}
		\pgfmathsetmacro\r{1}
		\draw[->,thick] (-1,0) --  (1,0);
            \node at (-\x, -\y) {$k$};
            \node at (\x, -\y) {$k+1$};
		\draw [thick] (\x, 0) arc (155:115:2*\r);
		\draw [thick] (\x,0) arc (175:125:2*\r);
		\draw [thick, -stealth reversed] (\x,0) arc (175:160:2*\r);
		\draw [thick] (-\x,0) arc (25:65:2*\r);
		\draw [thick, -stealth reversed] (-\x,0) arc (25:40:2*\r);
		\draw [thick] (-\x,0) arc (5:55:2*\r);
            \node[shape=circle,fill=darkgray, scale=0.5] at (-\x, 0) {};
            \node[shape=circle,fill=darkgray, scale=0.5] at (\x, 0) {};
	\end{tikzpicture}\right)
        +  w_\so\left(\begin{tikzpicture}[baseline={([yshift=-.5ex]current bounding box.center)}]
		\pgfmathsetmacro\x{0.5}
		\pgfmathsetmacro\y{0.3}
		\pgfmathsetmacro\r{1}
		\draw[->,thick] (-1,0) --  (1,0);
            \node at (0, -\y) {$k$};
		\draw [thick] (0,0) arc (10:60:2*\r);
		\draw [thick] (0,0) arc (170:120:2*\r);
		\draw [thick, -stealth reversed] (0,0) arc (170:155:2*\r);
            \draw [thick] (1.1*\r,0.8*\r) arc (60:120:2.2*\r);
            \node[shape=circle,fill=darkgray, scale=0.5] at (0, 0) {};
	\end{tikzpicture}\right)
        -  w_\so\left(\begin{tikzpicture}[baseline={([yshift=-.5ex]current bounding box.center)}]
		\pgfmathsetmacro\x{0.5}
		\pgfmathsetmacro\y{0.3}
		\pgfmathsetmacro\r{1}
		\draw[->,thick] (-1,0) --  (1,0);
            \node at (0, -\y) {$k$};
		\draw [thick] (0,0) arc (35:70:2*\r);
		\draw [thick, -stealth reversed] (0,0) arc (35:45:2*\r);
		\draw [thick] (0,0) arc (145:110:2*\r);
            \draw [thick] (\r,\r) arc (60:120:2*\r);
            \node[shape=circle,fill=darkgray, scale=0.5] at (0, 0) {};
	\end{tikzpicture}\right) + \\
        \\
        +  w_\so\left(\begin{tikzpicture}[baseline={([yshift=-.5ex]current bounding box.center)}]
		\pgfmathsetmacro\x{0.5}
		\pgfmathsetmacro\y{0.3}
		\pgfmathsetmacro\r{1}
		\draw[->,thick] (-1,0) --  (1,0);
            \node at (0, -\y) {$k$};
		\draw [thick] (0,0) arc (10:55:2*\r);
		\draw [thick, -stealth reversed] (0,0) arc (10:25:2*\r);
            \draw [thick] (0,0) arc (145:110:2*\r);
            \draw [thick] (\r,1.1\r) arc (80:125:2.5*\r);
            \node[shape=circle,fill=darkgray, scale=0.5] at (0, 0) {};
	\end{tikzpicture}\right)
        -  w_\so\left(\begin{tikzpicture}[baseline={([yshift=-.5ex]current bounding box.center)}]
		\pgfmathsetmacro\x{0.5}
		\pgfmathsetmacro\y{0.3}
		\pgfmathsetmacro\r{1}
		\draw[->,thick] (-1,0) --  (1,0);
            \node at (0, -\y) {$k$};
		\draw [thick] (0,0) arc (35:70:2*\r);
		\draw [thick] (0,0) arc (170:125:2*\r);
		\draw [thick, -stealth reversed] (0,0) arc (170:155:2*\r);
            \draw [thick] (\r,0.7\r) arc (55:100:2.5*\r);
            \node[shape=circle,fill=darkgray, scale=0.5] at (0, 0) {};
	\end{tikzpicture}\right)
\end{multline*}
\caption{The recurrence relation for the $\so$-weight system} \label{f-rec-so1}
\end{figure}
\end{itemize}

The first three graphs on the right are ordinary permutation graphs, which
are the same as in the defining relation for the $w_\fgl$ weight system. The last two graphs are extended permutation graphs rather than just permutation graphs. Their expression in terms of permutation graphs is determined by the global structure of the original graph and depends on whether the vertices~$k$ and~$k+1$ belong to the same cycle or to two different ones. 
\end{definition}

Note that the Casimir elements with the odd indices $C_1,C_3,C_5,\dots$
for the Lie algebra $\so(N)$ are also well-defined as the elements of the center~$ZU\so(N)$. But we have $C_1=0$, and the higher Casimirs $C_3,C_5,\dots$
can be expressed polynomially in terms of the Casimir elements $C_2,C_4,\dots$
with even indices.

It follows immediately from the definition that the universal $\so$
weight system remains unchanged under the cyclic shift of a
permutation.

\begin{theorem}
    The universal $\so$-weight system satisfies the generalized Vassiliev relations.
\end{theorem}

The proof of the theorem is similar to that of Theorem~\ref{t-glgV}.

Below, we will write the recurrence relation in 
Fig.~\ref{f-rec-so1} in the form
$$
w_\so(\alpha)=w_\so(\alpha')+w_\so(\beta_1)-w_\so(\beta_2)+w_\so(\gamma_1)-w_\so(\gamma_2).
$$
Here, permutations $\alpha$ and~$\alpha'$ are permutations of~$m$ elements, 
$\beta_1,\beta_2$ are permutations of~$m-1$ elements, and $ \gamma_1,\gamma_2$
are extended permutation graphs on~$m-1$  elements.

In spite of the fact that the $\so$-weight system distinguishes less 
than the $\gl$- one, it is not a specification of the latter.
In~\cite{KY24}, an example is given of a linear combination of order~$7$
chord diagrams, the value of $w_\so$ on which 
is nonzero, although it belongs to the kernel of $w_\gl$.

\subsection{Lie algebra $\so(N)$ with standard representation}

For the $\so(N)$ weight system in the standard representation, its relationship
with the numbers of faces of the corresponding hypermaps 
acquires the following form. Recall first the result of D.~Bar-Natan~\cite{BN95}.
A \emph{state}~$s$ of a chord diagram~$D$ is a mapping $s:V(D)\to\{1,-1\}$
of the set of chords of~$D$ into a two-element set. Each state determines
a two-dimensional surface with boundary; the surface in question
is obtained by attaching
a disc to the circle of~$D$ and replacing each chord by an ordinary  ribbon
attached to this disc,
if the state of the chord is~$1$, and by a half-twisted ribbon, if the state is~$-1$.
Denote by $f(D_s)$ the number of connected components of the boundary of the resulting surface,
and by~${\rm sign}(s)$ the \emph{sign} of~$s$, which is~$1$ if the number of chords in 
the state~$-1$ is even and $-1$ if this number is odd.

\begin{theorem}\cite{BN95}
For a chord diagram~$D$ of order~$n$,
$$w^{St}_{\so(N)}(D) = \sum_{s:\{1,2,\dots,n\}\to\{1,-1\}} {\rm sign}(s)
N^{f(D_s)-1}\ ,
$$
where the sum is taken over all $2^n$ states for $s$.
\end{theorem}

This theorem means that the result is valid for arc diagrams, which are 
involutions without fixed points.
To extend this result to arbitrary permutations, define a \emph{state} of 
a permutation $\alpha\in \BS_m$
as a mapping $s:\{1,2,\dots,m\}\to \{1,-1\}$. Note that according to this definition
a chord diagram of order~$n$ has $2^{2n}$ states as opposed to the $2^n$
states in original Bar-Natan's definition.

Similarly to the chord diagram case, 
each state $s$ determines a hypermap, which, like in the chord diagram case, is not necessarily
orientable. The boundary of the corresponding two-dimensional
surface with boundary is determined in terms of the graph of~$\alpha$
according to the following resolutions of each of the permuted elements:

$$
    
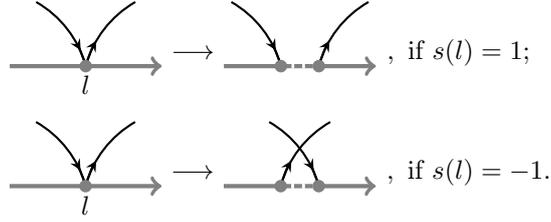
\begin{tikzpicture}[baseline={([yshift=-.5ex]current bounding box.center)}]
		\pgfmathsetmacro\x{0}
		\pgfmathsetmacro\y{0.5}
		\pgfmathsetmacro\r{0.7}
		\draw[->, ultra thick, gray] (-1,\y) --  (1,\y);
		\draw [thick] (\x,\y) arc (15:60:2*\r);
		\draw [thick, -stealth reversed] (\x,\y) arc (15:30:2*\r);
		\draw [thick] (-\x, \y) arc (165:120:2*\r);
		\draw [thick, -stealth] (-\x,\y) arc (165:150:2*\r);
            \node[shape=circle,fill=gray, scale=0.5] at (\x, \y) {};

            \node[] at (0, \y-0.25) {$l$};
            
            \node[] at (0, \y-0.4) {};
	\end{tikzpicture}
         \longrightarrow
        \begin{tikzpicture}[baseline={([yshift=-.5ex]current bounding box.center)}]
		\pgfmathsetmacro\x{0.25}
		\pgfmathsetmacro\y{0.5}
		\pgfmathsetmacro\r{0.7}
		\draw[ultra thick, gray] (-1,\y) --  (-\x,\y);
		\draw[->,ultra thick, gray] (\x,\y) --  (1,\y);
		\draw [ultra thick, gray] (-\x,\y) --  (\x,\y);
		\draw [ultra thick, white, dash pattern={on 1pt off 3pt}] (-\x,\y) --  (\x,\y);
		\draw [thick] (-\x,\y) arc (15:60:2*\r);
		\draw [thick, -stealth reversed] (-\x,\y) arc (15:30:2*\r);
		\draw [thick] (\x, \y) arc (165:120:2*\r);
		\draw [thick, -stealth] (\x,\y) arc (165:150:2*\r);
            \draw[ultra thick, dash pattern={on 1pt off 3pt}, white] (\x,\y) --  (\x,\y);
            \node[shape=circle,fill=gray, scale=0.5] at (\x, \y) {};
            \node[shape=circle,fill=gray, scale=0.5] at (-\x, \y) {};

            \node[] at (0, \y-0.4) {};

	\end{tikzpicture}\ ,\mbox{\ if\ }s(l)=1;
$$

$$
    \begin{tikzpicture}[baseline={([yshift=-.5ex]current bounding box.center)}]
		\pgfmathsetmacro\x{0}
		\pgfmathsetmacro\y{0.5}
		\pgfmathsetmacro\r{0.7}
		\draw[->, ultra thick, gray] (-1,\y) --  (1,\y);
		\draw [thick] (\x,\y) arc (15:60:2*\r);
		\draw [thick, -stealth reversed] (\x,\y) arc (15:30:2*\r);
		\draw [thick] (-\x, \y) arc (165:120:2*\r);
		\draw [thick, -stealth] (-\x,\y) arc (165:150:2*\r);
            \node[shape=circle,fill=gray, scale=0.5] at (\x, \y) {};

            \node[] at (0, \y-0.25) {$l$};
            
            \node[] at (-\x-1.18, \y-0.4) {};
            
	\end{tikzpicture}
         \longrightarrow
        \begin{tikzpicture}[baseline={([yshift=-.5ex]current bounding box.center)}]
		\pgfmathsetmacro\x{0.25}
		\pgfmathsetmacro\y{0.5}
		\pgfmathsetmacro\r{0.7}
		\draw[ultra thick, gray] (-1,\y) --  (-\x,\y);
		\draw[->,ultra thick, gray] (\x,\y) --  (1,\y);
		\draw [ultra thick, gray] (-\x,\y) --  (\x,\y);
		\draw [ultra thick, white, dash pattern={on 1pt off 3pt}] (-\x,\y) --  (\x,\y);
		\draw [thick] (\x,\y) arc (15:60:2*\r);
		\draw [thick, -stealth reversed] (\x,\y) arc (15:30:2*\r);
		\draw [thick] (-\x, \y) arc (165:120:2*\r);
		\draw [thick, -stealth] (-\x,\y) arc (165:150:2*\r);
            \draw[ultra thick, dash pattern={on 1pt off 3pt}, white] (-\x,\y) --  (\x,\y);
            \node[shape=circle,fill=gray, scale=0.5] at (\x, \y) {};
            \node[shape=circle,fill=gray, scale=0.5] at (-\x, \y) {};

            \node[] at (0, \y-0.4) {};

	\end{tikzpicture}\ ,\mbox{\ if\ }s(l)=-1.
$$

For a point in the state~$1$ (respectively, $-1$), the end of the entering arrow 
of the permutation graph 
shifts to the left (respectively, to the right), and the beginning 
of the leaving arrow shifts to the right (respectively, to the left)
along the horizontal line.

In other words, to a disjoint cycle of length~$\ell$ an $\ell$-hyperedge
is associated. The latter is a $2\ell$-gon, each second edge of which is attached
to the disc, either in a non-twisted, or in a half-twisted fashion, depending
on the state of the corresponding permuted element.
Denote by $f(\alpha_s)$ the number of connected components of the obtained curve
(that is, the number of connected components of the boundary of the resulting surface).
The \emph{sign} of~$s$, ${\rm sign}(s)$, is~$1$ if the number of legs in 
the state~$-1$ is even and $-1$ if this number is odd.

\begin{theorem}
For a permutation~$\alpha\in\BS_m$, we have
$$w^{St}_{\so(N)}(\alpha) = \sum_{s:\{1,2,\dots,m\}\to\{1,-1\}}{\rm sign}(s)
N^{f(\alpha_s)-1}\ ,
$$
where the sum is taken over all the $2^m$ states for $\alpha$.
\end{theorem}

\begin{proof} The weight system corresponding to
the standard representation of the Lie algebra $\so(N)$
is the composition of the mapping $\frac1{N}{\rm Tr}$ with 
the standard representation of $U\so(N)$.
The generator $F_{ij}$, $1\le i,j\le N$ is represented by the difference of
the matrix units $F_{ij}=E_{ij}-E_{\bar j\bar i}$. A product of matrix units $E_{ij}$
is either a zero matrix, or a matrix unit, and its trace is nonzero
iff it is a diagonal matrix unit. In the latter case, the trace is~$1$.

For a given state of a permutation of~$m$ elements, 
associate an index between~$1$ and~$N$ to each 
connected component of the boundary of the corresponding hypermap,
obtained as described above.
To such an assignment, a unique product of~$m$ matrix units $E_{ij}$ is associated:
the indices $i$ and $j$ at each leg are determined by the two boundary
components passing through this leg (which may well coincide).
Here, from the difference $F_{ij}=E_{ij}-E_{\bar j\bar i}$ we take $E_{ij}$ if the
corresponding leg is in state~$1$, and $-E_{\bar j\bar i}$ otherwise.
In this way we establish a one-to-one correspondence between states with
numbered boundary components and those monomials in the matrix units 
in the product~(\ref{eq:wsonsigma}) which yield a nonzero contribution
to the trace of the resulting sum of the matrices in the standard representation.

Now, since the index of each boundary component varies from~$1$ to~$N$,
the contribution to the trace of a given state~$s$ of a permutation~$\alpha$
is~$N^{f(\alpha_s)}$. Dividing by~$N$ and taking into account that the sign
of the contribution is plus or minus depending on the parity of
the number of negative states
of the legs, we obtain the desired.
\end{proof}

In particular, for standard cycles we immediately obtain by induction over
the number of permuted elements

\begin{corollary}  The standard representation of~$\so$ on the standard cycles
acquires the form

 $$
    w_{\so(N)}^{St}((1,2,\dots,m))= w_{\so(N)}^{St}
        \left(\begin{tikzpicture}[baseline={([yshift=-.5ex]current bounding box.center)}]
		\pgfmathsetmacro\x{0}
		\pgfmathsetmacro\y{0.5}
		\pgfmathsetmacro\r{1}
		\pgfmathsetmacro\t{0.5}
		\pgfmathsetmacro\e{0.25}
		\draw[->,thick] (\x-\t,\y) --  (\x+6*\t,\y);            

		\draw [thick] (\x, \y) arc (180:0:\t/2);
		\draw [thick, ->] (\x,\y) arc (180:80:\t/2);
		\draw [thick] (\x+\t, \y) arc (180:0:\t/2);
		\draw [thick, ->] (\x+\t, \y) arc (180:80:\t/2);
		\draw [thick] (\x+4*\t, \y) arc (180:0:\t/2);
		\draw [thick, ->] (\x+4*\t,\y) arc (180:80:\t/2);
		\draw [thick, ->] (\x+2*\t,\y) arc (180:80:\t/2);
		\draw [thick] (\x+4*\t,\y) arc (0:90:\t/2);
		\draw [thick] (\x, \y) arc (210:330:5*\t*0.57735);
		\draw [thick, -<] (\x,\y) arc (210:270:5*\t*0.57735);

            \node[shape=circle,fill=darkgray, scale=0.4] at (\x, \y) {};
            \node at (\x,\y-\e) {$1$};
            \node[shape=circle,fill=darkgray, scale=0.4] at (\x+\t, \y) {};
            \node at (\x+\t,\y-\e) {$2$};
            \node[shape=circle,fill=darkgray, scale=0.4] at (\x+2*\t, \y) {};
            \node at (\x+2*\t,\y-\e) {$3$};
            \node[shape=circle,fill=darkgray, scale=0.4] at (\x+4*\t, \y) {};
            \node[shape=circle,fill=darkgray, scale=0.4] at (\x+5*\t, \y) {};
            \node at (\x+5*\t+0.05,\y-\e) {$m$};

            \node at (\x+3*\t,\y+\e/2) {...};

            \node[scale=0.5] at (\x, 0) {};
	\end{tikzpicture}\right) 
    =\begin{cases}
     \frac{(N-1)^m+1-N}{N}, &\qquad m\text{ is odd};\\
     \frac{(N-1)^m-1+N^2}{N}, &\qquad m\text{ is even}.
      \end{cases}
$$
 In other words, the value of the~$\so(N)$ weight system in the standard representation
 on a given permutation~$\alpha$ can be obtained by substituting these
 values of Casimirs to $w_{ \so(N)}(\alpha)$.
\end{corollary}

Below, for~$k$ even, we will use notation
 \begin{equation}\label{e-P}
       P_k(N)=\frac{(N-1)^k-1+N^2}{N}
  \end{equation} 
for polynomials which are the values of the~$\so(N)$ weight system on standard
cycles of even length~$k$.

\section{Hopf algebras of hyper chord diagrams}\label{s5}

Chord diagrams, when considered modulo $4$-term relations, span a graded
commutative algebra. 
Below, we  show that modulo generalized Vassiliev
relations
permutations also span a graded commutative
algebra. In addition, there are two different natural ways to
introduce a comultiplication making the quotient vector space 
into a Hopf algebra. The first one of them is based on all possible
splittings of the set of cycles of a permutation into two disjoint
subsets, while the second one requires splitting the set of permuted
elements (legs) into two disjoint subsets in all possible ways.
In the present paper, we concentrate on the first comultiplication.
The corresponding  Hopf algebra is graded by partitions, 
so that the ordinary
Hopf algebra of chord diagrams is a Hopf subalgebra
in it spanned by the elements with gradings given by the partitions~$2^n$, $n=0,1,2,\dots$. Here and below we use multiplicative notation for partitions,
so that $\lambda=2^{\ell_2}3^{\ell_3}\dots$ denotes the partition
consisting of $\ell_2$ parts equal to~$2$, $\ell_3$ parts equal to~$3$, \dots.

Other Hopf algebras we consider are rotational Hopf algebras. 
Introduction of these Hopf algebras is motivated by the
fact that the Vassiliev relations, classical as well as generalized
ones, are not necessary to
define Lie algebra weight systems as functions on generalized chord diagrams.
Multiplication of chord diagrams and, more generally, arbitrary permutations,
is well defined if we consider chord diagrams modulo rotations.
Lie algebra weight systems are well defined on classes of rotational
equivalence, which makes rotational Hopf algebras worth
being studied  from the point of view of Lie algebra weight systems.

\subsection{Graded commutative cocommutative Hopf algebras}

For definiteness, the ground field is the field $\BC$ of complex numbers.
Most of the Hopf algebras considered in this paper are connected graded
commutative and cocommutative,
$$\cH=\cH_0\oplus\cH_1\oplus\cH_2\oplus\dots.$$
Here each $\cH_k$ is a finite-dimensional vector space, $\cH_0\equiv\BC$,
and there is a graded multiplication $\cH\otimes\cH\to\cH$ and
a graded comultiplication denoted by $\mu:\cH\to\cH\otimes\cH$,
consistent with the multiplication.
An element $p\in\cH$ of a Hopf algebra is \emph{primitive} if
$\mu(p)=1\otimes p+p\otimes1$. Primitive elements form vector
subspaces $P(\cH_k)\subset\cH_k$ in the homogeneous spaces~$\cH_k$,
$k=1,2,\dots$.
According to the Milnor--Moore theorem~\cite{MM65}, choosing a homogeneous basis
in the space of primitive elements establishes an isomorphism between this
Hopf algebra and the Hopf algebra of polynomials in the basic
primitive elements.

It follows that there is a well-defined projection $\pi:\cH_k\to P(\cH_k)$,
whose kernel is the space of \emph{decomposable} elements, which are polynomials
in elements of homogeneous degrees less than~$k$. 
In order to define a Hopf algebra homomorphism of a graded commutative
cocommutative Hopf algebra~$\cH$, it suffices to define a linear mapping
of the primitive elements $P(\cH)$ in~$\cH$.

\subsection{Hopf algebras modulo generalized Vassiliev relations}

Ordinary chord diagrams span a Hopf algebra~\cite{K93}.
Similarly, the vector space
spanned by hyper chord diagrams, when considered modulo generalized
Vassiliev relations, form a graded commutative cocommutative Hopf algebra.

Denote by~$\cH_m$, $m=0,1,2,3,\dots$, the vector space spanned by
hyper chord diagrams with~$m$ legs, modulo generalized Vassiliev relations defined in Sec.~\ref{ssgVr}.
 Similarly to the case of chord diagrams,
we define the \emph{product} of two hyper chord diagrams as the hyper chord
diagram represented by concatenation of any two hyper arc diagrams representing
the factors. Lemma~\ref{lrot} immediately implies

\begin{proposition}
Multiplication of hyper chord diagrams is well defined, that is,
the product of two hyper chord diagrams does not depend on
the choice of the hyper arc diagram presentations of the factors.
\end{proposition}

Let~$D$ be a hyper chord diagram, and let~$L(D)$ denote the set
of its legs. Each subset $R\subset L(D)$ determines a hyper chord
diagram $D|_R$, the \emph{restriction} of~$D$ to~$R$. 
If we interpret~$D$
as a permutation of~$L(D)$, then $D|_R$ is the permutation of~$R$,
which takes each element $r\in R$ to the next element in its orbit
that belongs to~$R$.
In particular, if $V(D)$ is the set of hyperchords of~$D$
and $U\subset V(D)$ is a subset, then $D|_U$ is the hyperchord
diagram consisting of hyperchords belonging to the subset~$U$.

The \emph{coproduct}
$\mu(D)$ of a hyper chord diagram~$D$ is defined as
$$
\mu(D)=\sum_{U\sqcup W=V(D)}D|_U\otimes D|_W,
$$
where the summation is carried over all decompositions of the
set of hyperedges of~$D$ into disjoint union of two subsets.

\begin{remark}
    There is another natural way to introduce a coproduct on
    the space of hyperchord diagrams. Namely, we may define it
    as a mapping $D\mapsto \sum_{L(D)=L_1\sqcup L_2}D|_{L_1}\otimes
    D|_{L_2}$, where the summation is carried over all splittings
    of the set of legs into two disjoint subsets. Below,
    we briefly discuss this second 
    comultiplication in the context of rotational
    Hopf algebras.
\end{remark}

Standard argument yields

\begin{proposition}
The product and coproduct defined above induce on the direct sum
$$
\cH=\cH_0\oplus\cH_1\oplus\cH_2\oplus\dots
$$
a structure of a graded commutative cocommutative Hopf algebra.
\end{proposition}

For a permutation $\alpha\in \BS_m$, denote by $\lambda(\alpha)\vdash m$
the partition of~$m$ given by the set of lengths of the cycles in
the representation of~$\alpha$ in the product of disjoint cycles.
All the terms in a generalized Vassiliev relation for a given
hyper chord diagram determine one and the same partition,
whence each vector space $\cH_m$  splits into a direct sum
$$\cH_m=\bigoplus_{\lambda\vdash m}\cH_\lambda$$
of subspaces spanned by hyper chord diagrams with a given cycle structure~$\lambda$.
Ordinary chord diagrams span a Hopf subalgebra of~$\cH$, whose homogeneous
subspaces correspond to partitions $2^0,2^1,2^2,2^3,\dots$.

Similarly to the case of the Hopf algebra of chord diagrams~\cite{L97},
the projection to the subspace of primitive elements~$P(\cH_m)\subset\cH_m$
whose kernel is the subspace of decomposable elements~$D(\cH_m)\subset\cH_m$
can be given by the formula
$$
\pi:\alpha\mapsto\alpha-1!\sum_{I_1\sqcup I_2}\alpha|_{I_1}\alpha|_{I_2}
+2!\sum_{I_1\sqcup I_2\sqcup I_3}\alpha|_{I_1}\alpha|_{I_2}\alpha|_{I_3}-\dots,
$$
where the sums on the right run over all partitions of the set $V(\alpha)$
of the set of cycles in~$\alpha$ into an unordered disjoint union of
a given number of nonempty subsets.
Hence, this projection preserves the cycle structure of a permutation, and the space $P(\cH_m)$
splits as the direct sum
$$
P(\cH_m)=\bigoplus_{\lambda\vdash m} P(\cH_\lambda).
$$

It follows then that any tuple (either finite or infinite) $\nu$ of
pairwise distinct positive integers determines a Hopf subalgebra of~$\cH$
spanned by permutations in which the length of each disjoint cycle 
belongs to~$\nu$. We denote this Hopf subalgebra by~$\cH(\nu)$.
The Hopf subalgebra $\cH(2)\subset\cH$ coincides with the ordinary Hopf algebra
of chord diagrams.


Table~\ref{tab1} contains the dimensions of the spaces~$\cH_\lambda$
and $P(\cH_\lambda)$
for all partitions~$\lambda$ of small numbers.
Note that, for~$m\ge2$, each permutation in~$\BS_m$ having cycles of length~$1$
(`fixed points') is a product of two permutations of smaller orders.
Therefore, its projection to the subspace of primitives is~$0$, whence
adding parts equal to~$1$ to a nonempty partition~$\lambda$
does not change the dimension of the space of primitive elements $P(\cH_\lambda)$.
That is why we do not insert these values in the tables.

\begin{table}
\begin{center}
	\begin{tabular}{| c | c | c | c | c | c |} 
 \hline
$m$& $\dim~\cH_m$&$\dim~P(\cH_m)$&$\lambda\vdash m$&$\dim~\cH_\lambda$&$\dim~P(\cH_\lambda)$\\  
\hline\hline 1 & 1 & 1 &  &  & \\ 
\hline &  &  & $1^1$ & 1 & 1\\ 
\hline\hline 2 & 2 & 1 & & & \\ 
\hline &  &  & $2^1$ & 1 & 1\\ 
\hline\hline 3 & 4 & 2 & & & \\ 
\hline &  &  & $3^1$ & 2 & 2\\ 
\hline\hline 4 & 9 & 4 &  &  & \\ 
\hline &  &  & $4^1$ & 3 & 3\\ 
\hline &  &  & $2^2$ & 2 & 1\\ 
\hline\hline 5 & 21& 10 &  &  & \\ 
\hline &  &  & $5^1$ & 8 & 8 \\ 
\hline &  &  & $2^13^1$ & 4 & 2\\ 
\hline\hline 6 & 66 & {37} &  &  & \\ 
\hline &  &  & $6^1$ & 24  & 24\\ 
\hline &  &  & $2^14^1$ & 9 & 6\\ 
\hline &  &  & $3^2$ & 9 & 5\\ 
\hline &  &  & $2^3$ & 3 & 1\\ 
\hline\hline 7 & 242 & 156 &  &  & \\ 
\hline &  &  & $7^1$ & 108 & 108\\ 
\hline &  &  & $2^15^1$ & 27 & 19\\ 
\hline &  &  & $3^14^1$ & 31 & 25\\ 
\hline &  &  & $2^23^1$ & 10 & 4\\ 
\hline\hline 8 & 1239 & 922 &  &  & \\ 
\hline &  &  & $8^1$ & 640 & 640\\ 
\hline &  &  & $2^16^1$ &111  & 87\\ 
\hline &  &  & $3^15^1$ &126  & 110\\ 
\hline &  &  & $4^2$ & 68 & 59\\ 
\hline &  &  & $2^13^2$ & 22 & 9 \\ 
\hline &  &  & $2^24^1$ & 24 & 9\\ 
\hline &  &  & $2^4$ & 6 & 2\\ 
\hline\hline 9 & 7832 & 6299 &  &  & \\ 
\hline &  &  & $9^1$ & 4492 & 4492\\ 
\hline &  &  & $2^17^1$ & 586  & 478\\ 
\hline &  &  & $3^16^1$ & 647 & 599\\ 
\hline &  &  & $4^15^1$ & 637 &  613\\ 
\hline &  &  & $2^25^1$ & 77 & 28\\ 
\hline &  &  & $2^13^14^1$ & 95 & 46\\ 
\hline &  &  & $3^3$ & 39 & 21\\ 
\hline &  &  & $2^33^1$ & 20 & 6\\ 
\hline
\end{tabular}
\end{center}
\caption{Table of the values of dimensions of homogeneous subspaces $\cH_m$
	and partition-homogeneous subspaces $\cH_\lambda$ 
	and subspaces of primitive elements in them in the Hopf algebra~$\cH$,
	having no parts equal to~$1$ (with the exception of~$\cH_{1^1}$), for small values of~$m$}\label{tab1}
\end{table}

Computation of the dimensions of the vector spaces 
$\cH_{2^n}$, which are the vector spaces of chord diagrams modulo
$4$-term relations, is known to be a very complicated problem.
These homogeneous vector spaces are indexed by partitions with
the largest possible number of parts. In contrast, if the number
of parts is~$1$, so that we consider cyclic permutations of
given length, the answer is well-known and relatively easy.
The vector space $\cH_{k^1}$ is spanned by hyper chord diagrams with the cycle structure $k^1$. 

\begin{theorem}
    The dimensions of the vector spaces $\cH_{m^1}$
    spanned by permutations with a single cycle of length~$m$ are 
    $$\dim (\cH_{m^1})=\frac{1}{m^2}\sum_{d|m} \phi(d)^2\cdot\left(\frac{m}{d}\right)!d^{\frac{m}{d}},$$
    where $\phi$ is the Euler function. 
\end{theorem}

Indeed,
    for permutations of this type, two-hyperchord
    generalized Vassiliev
    relations are empty, while one-hyperchord ones     
    are reduced to the relations identifying
    a permutation with its cyclic shift one.
    As a result, $\dim~(\cH_{m^1})$ coincides with the number
    of orbits of the cyclic group~$\BZ$ action on the set
    of cyclic permutations by rotations, enumerated e.g. 
    in~\cite{Yo14}.

In particular, we may give a reasonable lower bound for
the dimensions of these spaces:

\begin{theorem}
    The dimension of vector space $\cH_{k^1}$ can be estimated from below as follows
    $$\dim (\cH_{m^1}) \ge \frac{(m-1)!}{m}.$$
    This estimation provides an asymptotics for the
    sequence $\dim (\cH_{m^1})$.
\end{theorem}
Indeed, consider the set of cyclic permutations of $m$ elements; the cardinality of such set is equal to $(m-1)!$. Each cyclic permutation corresponds to some hyper chord diagram in $\cH_{m^1}$. On the other side, each hyper chord diagram admits at most $m$ representations as a hyper arc diagram.

For~$m$ prime, all the orbits of the cyclic group $\BZ_m$
acting on cyclic permutations by shifts but $m-1$ are 
of length~$m$. The exceptional orbits
have length~$1$ and consist of  cycles consisting
of steps of constant lengths.
Therefore, for prime values of~$m$, we have
$\dim (\cH_{m^1})=\frac{(m-1)!-(m-1)}m+m-1$.

\subsection{Rotational Hopf algebras}

The $4$-term relations for chord diagrams allow for defining 
multiplication of chord diagrams: concatenation of arc diagrams,
which depends on the choice of an arc diagram representing a given chord
diagram, becomes independent when taken modulo $4$-term relations.
In particular, $4$-term relations imply that any two arc presentations
of a chord diagram are equivalent.

However, studying $\gl$- and $\so$-weight systems on chord diagrams
and permutations does not necessarily require introduction of (generalized)
Vassiliev relations.
From the point of view of constructing Hopf algebras and weight systems,
$4$-term relations may be weakened. Namely, it suffices to consider 
arc diagrams and their products modulo rotational equivalence.
A similar construction is applicable to permutations (hyper arc diagrams)
of various kinds. The latter Hopf algebras naturally appeared 
in~\cite{KL25,KKL24,LY24}.

\subsubsection{Rotational Hopf algebras of permutations, positive 
permutations, monotone permutations, and chord diagrams}

We say that a permutation~$\alpha\in\BS_m$ is \emph{topologically connected} 
(or just \emph{connected}) if there
is no proper subsegment $[k+1,k+2,\dots,n]\subsetneq[1,\dots,m]$ invariant 
under the action of~$\alpha$.

Now let us introduce the notion of rotational equivalence on permutations.

 Two connected permutations $\alpha_1,\alpha_2\in \BS_m$
are said to be \emph{rotationally equivalent} if each of them coincides with
a cyclic shift of the other one; in other words, if they are conjugate by
some power of the standard cyclic permutation $\sigma\in\BS_m$:
there is~$k\in\BZ$ such that $\alpha_2=\sigma^{-k}\alpha_1\sigma^k$.
To each permutation~$\alpha\in\BS_m$, 
one associates a unique tuple of connected permutations the sum of whose lengths
is~$m$, which can be defined recursively in the following way.
If $\alpha$ is connected, then the tuple is~$\{\alpha\}$.
Otherwise, there is a proper subsegment $[k+1,k+2,\dots,n]\subsetneq[1,\dots,m]$
invariant with respect to~$\alpha$, and the tuple associated to $\alpha$
is the union of two tuples, the one associated to $\alpha|_{[k+1,k+2,\dots,n]}$, and another one associated to $\alpha|_{[1,2,\dots,m]\setminus[k+1,k+2,\dots,n]}$. 
Two arbitrary, not necessarily connected, permutations are said to be
\emph{rotationally equivalent}, if there is a one-to-one
mapping between the  tuples of connected permutations
associated to each of them, which takes each connected permutation to a rotationally
equivalent one.
Thus, the tuple of connected permutations associated to a given permutation,
determines uniquely its rotational equivalence class.

The \emph{rotational product} of two rotational equivalence classes of
 permutations is  the rotational equivalence class of a permutation, which 
 is determined by the tuple, which is the union of two tuples associated
 to the factors.
By the \emph{rotational algebra of permutations} we mean the commutative algebra~$\cA$
freely generated by rotational equivalence classes of connected permutations,
$$
\cA=\cA_0\oplus\cA_1\oplus\cA_2\oplus\dots,
$$
where each homogeneous space~$\cA_m$ is spanned by rotational equivalence classes
of permutations of~$m$ elements.
For example, $\cA_3=\langle(1)(2)(3),(1)(2,3),(1,2,3),(1,3,2)\rangle$, so that
$\dim\cA_3=4$.

The rotational algebra~$\cA$ of permutations is endowed with a natural comultiplication
$\mu:\cA\to\cA\otimes\cA$ defined on the generators as follows
$$
\mu:\alpha\mapsto \sum_{U\sqcup W=V(\alpha)} \alpha|_U\otimes \alpha|_W,
$$
where the summation is carried over all ordered disjoint partitions of the
set $V(\alpha)$ of disjoint cycles in~$\alpha$ into two subsets, and $\alpha|_U$
denotes the restriction of~$\alpha$ to the subset of elements entering
the cycles in~$U$.

An~\emph{accent} in a permutation $\alpha\in\BS_m$ is an element $i\in\{1,\dots,m\}$
such that $\alpha(i)>i$. Let~$a(\alpha)$ denote the number of accents in~$\alpha$.
Obviously, the number of accents is preserved under cyclic equivalence of permutations.
We say that a permutation $a(\alpha)\in\BS_m$ is \emph{positive} (respectively, negative)
if $a(\alpha)=m-c(\alpha)$ (respectively, if $a(\alpha)=c(\alpha)$).
This is a maximal (respectively, minimal) 
number of accents in a permutation of~$m$ elements having
a given number of disjoint cycles. 
A permutation is positive (respectively, negative) iff
each disjoint cycle in it is positive (respectively, negative).
It follows that each subpermutation of a positive (respectively, negative) permutation
is positive (respectively, negative).

The following Hopf subalgebras arise naturally in the study of Lie algebra weight systems
on permutations~\cite{KKL24}.

Positive (respectively, negative) permutations generate a Hopf subalgebra of~$\cA$.
We denote these Hopf algebras by~$\cA^+$ (respectively, $\cA^-$). We also consider the
Hopf subalgebra $\cA^\pm$ of \emph{monotone} permutations; it is generated by permutations
in which each disjoint cycle is either positive or negative, so that both $\cA^+$
and $\cA^-$ are Hopf subalgebras in $\cA^\pm$. 

Similarly to the case of the Hopf algebra~$\cH$,
any (not necessarily finite)
tuple~$\nu$ of pairwise distinct positive integers determines a Hopf subalgebra 
$\cA(\nu)$ of the Hopf algebra~$\cA$; this Hopf subalgebra is generated
by rotational equivalence classes of permutations the length of each disjoint
cycle in which belongs to~$\nu$. We also consider the Hopf subalgebras
$\cA^+(\nu)=\cA^+\cap\cA(\nu)$, $\cA^-(\nu)=\cA^-\cap\cA(\nu)$,
$\cA^\pm(\nu)=\cA^\pm\cap\cA(\nu)$. 
In particular, the Hopf subalgebra $\cA(2)=\cA^+(2)=\cA^-(2)=\cA^\pm(2)$
is generated by rotational equivalence classes of chord diagrams.
We denote this Hopf algebra by~$\cC$. Note that it is different from the
conventional Hopf algebra of chord diagrams, in which the $4$-term relations
are factored out.

Tables~\ref{tab2}, \ref{tab3}, \ref{tab4} contain dimensions of homogeneous spaces
in restricted rotational Hopf algebras for small values of the gradings.

\begin{table}
\begin{center}
	\begin{tabular}{| c | c | c | c | c | c |} 
 \hline
$m=2n$& $\dim~\cA_m(2)=\dim~\cA^+_m(2)$&$\dim~P(\cA_m(2))=\dim~P(\cA^+_m(2))$\\  
\hline\hline 2 &  1& 1 \\ 
\hline 4 & 2 & 1  \\ 
\hline 6 & 4 & 2  \\ 
\hline 8 & 11 & 6 \\ 
\hline
\end{tabular}
\end{center}
\caption{Table of the values of dimensions of homogeneous subspaces $\cA_m(2^n)$
and of primitive subspaces in it, for small values of~$m$}\label{tab2}
\end{table}

\begin{table}
\begin{center}
	\begin{tabular}{| c | c | c | c | c | c |} 
 \hline
$m=3n$& $\dim~\cA_m(3)$&$\dim~P(\cA_m(3))$& $\dim~\cA^+_m(3)$&$\dim~P(\cA^+_m(3))$\\  
\hline\hline 3 &  2& 2 &1 &1\\ 
\hline 6 &  10& 7 &3 &2\\
\hline 9 &  219& 201 & 28 & 25\\
\hline
\end{tabular}
\end{center}
\caption{Table of the values of dimensions of homogeneous subspaces $\cA_m(3^n)$
and of primitive subspaces in it, for small values of~$m$}\label{tab3}
\end{table}

\begin{table}
\begin{center}
	\begin{tabular}{| c | c | c | c | c | c |} 
 \hline
$m$& $\dim~\cA_m(2,3)$&$\dim~P(\cA_m(2,3))$& $\dim~\cA^+_m(2,3)$&$\dim~P(\cA^+_m(2,3))$\\  
\hline
\hline 1 &  0& 0 & 0 &0\\ 
\hline 2 &  1& 1 & 1 &1\\ 
\hline 3 &  2& 2 & 1 &1\\
\hline 4 &  2& 1 & 2 &1\\
\hline 5 &  4& 2 & 2 &1\\
\hline 6 &  14& 9 & 7 &4\\
\hline 7 &  18& 12 & 9 &6\\
\hline 8 &  118& 99 & 39 &30\\
\hline 9 &  353& 311 & 95 &80\\
\hline
\end{tabular}
\end{center}
\caption{Table of the values of dimensions of homogeneous subspaces $\cA_m(2^n3^k)$
and of primitive subspaces in it, for small values of~$m$}\label{tab4}
\end{table}

\subsubsection{Dimensions of the homogeneous subspaces in the 
rotational Hopf algebras}

The Hopf algebra~$\cA$ defined in the previous section is graded,
$$
\cA=\cA_0\oplus\cA_1\oplus\cA_2\oplus\dots,
$$
where~$\cA_m$ is the vector space freely spanned by rotational equivalence classes
of permutations of~$n$ elements. This grading induces gradings on each of the 
Hopf subalgebras $\cA^+,\cA^-,\cA^\pm$, as well as their $\nu$-specializations.

The problem of computation of the dimensions of the homogeneous subspaces
is complicated, due to the cyclic equivalence relations we impose.
However, certain precise values as well as estimates for the dimensions
can be obtained relatively easily. In particular, the isomorphism $\cA\to\cA$
taking a permutation to its inverse (that is,
reversing orientations of each of the disjoint cycles in permutations)
establishes an isomorphism $\cA^+\to\cA^-$, whence, obviously,
$\dim~\cA^+_m=\dim~\cA^-_m$, $m=1,2,3,\dots$. Below, we will restrict ourselves
with dimensions $\dim~\cA^+_m$.

Now we will make use of the one-to-one correspondence between partitions
of the set $\{1,\dots,m\}$ and positive permutations in~$\BS_m$. 
The numbers of set partitions are well known. The generating function for them
can be represented either in the form of a continuous fraction or as an
explicit sum of rational functions
\begin{eqnarray*}
P(x)=\\1 + x + 2 x^2 + 5 x^3 + 15 x^4 + 52 x^5 + 203 x^6 + 877 x^7 + 
 4140 x^8 + 21147 x^9 + 115975 x^{10}+\dots=\\
\frac{1}{1-x-\frac{x^2}{1-2x-\frac{2x^2}{1-3x-\frac{3x^2}{1-4x-\dots}}}}=\\
1 +\frac {x}{(1 - x)} + \frac{x^2}{(1 - x) (1 - 2 x)} + 
 \frac{x^3}{(1 - x) (1 - 2 x) (1 - 3 x)} + 
\frac{x^4}{(1 - x) (1 - 2 x) (1 - 3 x) (1 - 4 x)} + \dots
\end{eqnarray*}

For a tuple (not necessarily finite) $\nu=(\nu_1,\nu_2,\dots)$
of distinct positive integers $\nu_1<\nu_2<\nu_3<\dots$,
the generating series enumerating set partitions with parts of sizes 
$\nu_1,\nu_2,\nu_3,\dots$ only can be given in the form similar
to the last expression above:
$$
P_\nu(x)=1 +\frac {x^{\nu_1}}{(1 -\nu_1 x)} + \frac{x^{\nu_2}}{(1 -\nu_1 x) (1 - \nu_2 x)} + 
 \frac{x^{\nu_3}}{(1 -\nu_1 x) (1 - \nu_2 x) (1 -\nu_3 x)} +\dots,$$
which underlines the fact that the Hopf algebra $\cA^+$ coincides with $\cA^+(\nu)$,
for $\nu={\mathbb N}$, the set of positive integers.

The following statement is well-known and can be easily proved.

\begin{theorem}
    The generating function
$$
A(x)=x+x^2+x^3+2x^4+6x^5+21x^6+85x^7+\dots
$$    
enumerating connected positive permutations is the solution of the following
functional equation:
$$
P(x)=1+A(xP(x)).
$$
    The generating function $A_\nu(x)$
enumerating connected positive permutations of profile~$\nu$ 
is the solution of the following
functional equation:
$$
P_\nu(x)=1+A_\nu(xP_\nu(x)).
$$
\end{theorem}

Knowing the number of connected positive permutations of a number~$m$ of elements,
which is prime, immediately yields the number of generators  of degree~$m$ in~$\cA^+$.
For example, there are $85$ connected partitions of~$7$ elements.
Hence, there are $(85-1)/7+1=13$ generators of degree~$7$ in~$\cA^+$.
Indeed, the group of residues modulo~$7$ acts on the set of connected
positive permutations
of~$7$ elements by rotations; this action has a unique fixed point, namely,
the standard cycle of length~$7$, while all the other orbits are of length~$7$.

\subsubsection{The legwise comultiplication}

In addition to comultiplication~$\mu$, one can consider another coproduct~$\mu'$
of a permutation $\alpha\in\BS_m$ defined as follows:
$$\mu':\alpha\mapsto \sum_{A\sqcup B=\{1,2,\dots,m\}}\alpha|_A\otimes\alpha|_B,$$
where the summation is carried over all ordered partitions of the set
of legs $\{1,2,\dots,m\}$. This coproduct, when applied to rotational equivalence
classes of permutations, defines a new Hopf algebra structure on the rotational
algebras; we denote the corresponding Hopf algebras $\cA', \cA'^+ $, and so on.

In spite of the fact that the comultiplication $\mu'$ differs seriously from~$\mu$,
the corresponding Hopf algebras are very similar. As an example, let us reproduce a theorem
from~\cite{KKL24}. Since only a sketch of a proof is given there, we present here a complete proof.

Denote by~$\pi':\cA'\to P(\cA')$ the projection to the
subspace of primitives associated to the comultiplication~$\pi'$, so that
\begin{equation}\label{e-piprim}
    \pi':\alpha\to\sum_{k=1}^m(-1)^{k-1}(k-1)!
    \sum_{L_1\sqcup\dots\sqcup L_k=\{1,\dots,m\},L_i\ne \emptyset}
    \alpha|_{L_1}\dots\alpha|_{L_k},
\end{equation}
for $\alpha\in\BS_m$; here the second summation is carried
over all partitions of the set of legs of~$\alpha$
into nonempty subsets.

\begin{theorem}
If an arc diagram~$B\in\cA(2)$ contains more than one arc, then
its projections to primitives, $\pi(B)$ and $\pi'(B)$ in the two Hopf algebras
$\cA$ and $\cA'$ coincide, $\pi(B)=\pi'(B)$.
\end{theorem} 

In other words, for an arc diagram~$B$ with more than one arc the projection
$\pi'(B)$ is a linear combination of involutions without fixed points,
so that the coefficients of the permutations with fixed points vanish.

{\bf Proof.} 
We need to show that the expression~(\ref{e-piprim}) for $\pi'(B)$ contains no summands with cycles of length one for any chord diagram $B$ with more than one chord. Pick a chord $a\in V(B)$,
and let $a_1,a_2$ be its legs. 
It suffices to show that the contribution of all partitions of the legs $L(B)$ such that $a_1$ and $a_2$ belong to different parts vanishes. Any partition of the set of legs $L(B)$ of $B$ induces a partition of the set 
of legs of the chord diagram 
$B'_a=B\setminus \{a\}$, and the set of partitions
of $L(B)$ splits into a disjoint union
of subsets corresponding to partitions of $L(B'_a)$
into nonempty subsets.
Pick such a partition of 
$L(B'_a)$ and assume that it consists of $k$ parts. Then, the corresponding partition of the legs $L(B)$ of $B$
arises in
one of the following alternative cases:
\begin{itemize}
\item   the legs $a_1$ and $a_2$ are added to two 
different parts of the corresponding partition of
the complementary set of legs $L(B'_a)$;
\item one of the legs $a_1$, $a_2$ is added to a part
of the partition of $L(B'_a)$, while the other one forms
a separate part;
\item each of the legs $a_1$, $a_2$ forms a separate part.
\end{itemize}
The total number of parts in the corresponding partition
of $L(B)$ is, respectively, $k$, $k+1$, and $k+2$.
The sum of the coefficients in~(\ref{e-piprim}), up to a sign, is then
\begin{equation*}
    k(k-1)\cdot(k-1)!-2k\cdot k!+(k+1)!=
    (k^2-k-2k^2+k^2+k)\cdot(k-1)!=0,
\end{equation*}
with the exception of the case $k=0$, which arises
if~$B$ consists of a single chord.

It would be very interesting to know similar theorems
 for other part-specific Hopf subalgebras,
for example, for $\cA(3)$.

\subsubsection{Lie algebras weight systems and the rotational Hopf algebras}

In~\cite{KKL24,LY24}, Hopf algebra homomorphisms were associated
to the $\gl$- and $\so$-weight systems, respectively.
These homomorphisms are defined by means of a perturbation of 
the trace of the standard representations of the Lie algebras.

For the Lie algebra~$\gl(N)$, the trace in the standard representation yields
$\alpha\mapsto N^{f(\alpha)-1}$, see~\cite{KL25} and Sec.~\ref{ss42} above.
In particular, in the standard representation, we have $C_k=N^{k-1}$.
The substitution $C_k=p_k N^{k-1}$, $k=1,2,3,\dots$, 
is called the \emph{prechromatic substitution} in~\cite{KKL24}.
The coefficient of $N^0$ in the image of the mapping 
$$
\alpha\mapsto N^{c(\alpha)-m}w_\gl(\alpha)|_{C_k=p_k N^{k-1}},
$$
for $\alpha\in\BS_m$,
defines a mapping, which we denote by $X_0$, from permutations to polynomials in $p_1,p_2,\dots$.
The following theorem is proved in~\cite{KKL24}:

\begin{theorem}\label{t-Hhgl}
The mapping~$X_0$ extends to a Hopf algebra homomorphism $X_0:\cA\to\BC[p_1,p_2,\dots]$
from the rotational Hopf algebra of permutations to the Hopf algebra of polynomials in
$p_1,p_2,\dots$.
\end{theorem}

Note that~$X_0$ is a filtered rather than graded Hopf algebra homomorphism.
Indeed, it takes a permutation of~$m$ elements to a polynomial of
weighted degree $\le m$, where the weight of a variable $p_i$ is set to be~$i-2$.

Of course, the restrictions of~$X_0$ to the rotational Hopf subalgebras~$\cA^+,\cC,\dots$
of~$\cA$ also are Hopf algebra homomorphisms.

\begin{example}
    Cyclic permutations in~$\BS_m$ are primitive elements in the space~$\cA_m$.
    The mapping~$X_0$ takes a cyclic permutation~$\alpha$ with the number of
    accents~$a(\alpha)=m-k$ to the following linear combination of variables~$p_i$:
    $$
    X_0:\alpha\mapsto p_{m}-{k-1\choose1}p_{m-1}+{m-1\choose2}p_{m-2}-\dots
    +(-1)^{k-1}{k-1\choose k-1}p_{m-k+1}.
    $$
\end{example}

The substitution $p_1=p_2=\dots=x$  makes $X_0$ into the chromatic polynomial:

\begin{theorem}\cite{KKL24}
	For a positive permutation~$\alpha$, we have 
	$$
	X_0(\alpha)|_{p_1=p_2=\dots=x}=\chi_{\gamma(\alpha)}(x),
	$$
where $\gamma(\alpha)$ is the \emph{intersection graph} of~$\alpha$,
and $\chi$ denotes the chromatic polynomial. 
\end{theorem}

Here by the \emph{intersection graph} $\gamma(\alpha)$ of a permutation~$\alpha$
we mean the graph whose set of vertices is $V(\alpha)$, the set of disjoint
cycles in~$\alpha$, and two vertices are connected by an edge iff
the two cycles to which they correspond interlace in~$\alpha$, that is,
the restriction of~$\alpha$ to the union of these two cycles is a connected
permutation. For arc diagrams, this definition coincides with the ordinary one.

Note that for a nonpositive permutation, the indicated substitution yields~$0$.
For nonpositive long cycles, this assertion follows from the above example.

A similar statement is valid for the universal $\so$-weight system.
Namely, the trace in the standard representation of the Lie algebra $\so(N)$
yields the substitution~(\ref{e-P})
$$
C_k=\frac{ (1+N)^k-(1+N)}N=P_k(N),
$$
for $k$ even. Define the value of the mapping~$Y$ 
for a generator $\alpha\in\BS_m$, by the formula
$$
Y:\alpha\mapsto N^{c(\alpha)-m} w_\so(\alpha)|_{C_k=p_kP_k}.
$$
Let~$Y_0(\alpha)$ be the coefficient of~$N^0$ in~$Y(\alpha)$.
The following statement is an $\so$-analogue of Theorem~\ref{t-Hhgl}:

\begin{theorem}~\cite{LY24}
The mapping~$Y_0$ extends to a Hopf algebra homomorphism $Y_0:\cA^\pm\to\BC[p_2,p_4,\dots]$
from the rotational Hopf algebra of monotone permutations to the Hopf algebra of polynomials in
$p_2,p_4,\dots$.
\end{theorem}

Of course, the restriction of~$Y_0$ to the Hopf subalgebras $\cA^+,\cA^-,\cC$ of,
respectively, positive permutations, negative permutations, and chord diagrams
also is a Hopf algebra homomorphism.

\begin{theorem}~\cite{LY24}
	The substitution $p_k=x$, $k=2,4,6,\dots$, makes~$Y_0$ into the chromatic
	polynomial of the intersection graph of the monotone permutation~$\alpha$,
	up to a constant factor, which equals $2^{c_2(\alpha)}$, where~$c_2(\alpha)$
	is the number of cycles of length~$2$ in~$\alpha$.
\end{theorem}


\begin{thebibliography}{99}










\bibitem{BN95}
   {D.~Bar-Natan},
  {\it On the Vassiliev knot invariants},
  Topology,
   {1995},
  {\bf 34},
   {423--472}







		
		
		








\bibitem{CDBook12}
  {S.~Chmutov, S.~Duzhin, and J.~Mostovoy},
  {\it Introduction to Vassiliev Knot Invariants},
  Cambridge University Press,
    {2012}




\bibitem{CKL20}
S.~Chmutov, M.~Kazarian, and S.~Lando,
{\it Polynomial graph invariants and the KP
hierarchy},
Selecta Math. New Series
(2020)
26:34














\bibitem{DK07}
С.~В.~Дужин, М.~В.~Карев,
{\it 	Определение ориентации струнных зацеплений при помощи инвариантов конечного типа},
Функц. анализ и его прил., 2007, 41:3,  48--59











\bibitem{J71}
A.-A.A.Jucys
{\it Symmetric polynomials and the center of the symmetric group ring},
Reports on Mathematical Physics 5, no.~1 (1974), 107--112

\bibitem{KP70}
B.~B.~Kadomtsev, V.~I.~Petviashvili,
\emph{On the stability of solitary waves in weakly dispersive media}, Sov. Phys. Dokl. 15: 539--541 (1970).

\bibitem{KKL24} M.~Kazarian, N.~Kodaneva, S.~Lando
{\it The universal $\gl$-weight system and the chromatic polynomial}, 
arXiv:2406.10562

\bibitem{KL15} M.~Kazarian, S.~Lando, {\it Combinatorial solutions to integrable hierarchies}, Russ. Math. Surveys, {\bf 70} (3) 453--482 (2015).

\bibitem{KL22} M.~Kazarian, S.~Lando, {\it Weight systems and invariants of graphs and embedded graphs}, Russ. Math. Surveys, {\bf 77} (5) 893--942 (2022).

\bibitem{KY24} M.~Kazarian, Z.~Yang
{\it Universal Polynomial $\so$-Weight System},
arXiv:2411.11546


\bibitem{KL25}
N.~Kodaneva, S.~Lando
{\it Polynomial graph invariants induced from the $\gl$-weight system},
GEOPHY105421 (2025)

\bibitem{K93}
M.~Kontsevich.
\newblock {\it Vassiliev knot invariants},
\newblock in: {\em Advances in Soviet Math.}, 16(2):137--150, 1993.









\bibitem{L97}
  {S.~Lando},
  {\it On primitive elements in the bialgebra of chord diagrams},
   Translations of the American Mathematical Society-Series 2,
  {\bf 180},
   {167--174},
   {1997},
{Providence [etc.] American Mathematical Society, 1949-}




\bibitem{LZ03}
   {S.~Lando and A.~Zvonkin},
  {\it Graphs on Surfaces and Their Applications},
   {2003}



\bibitem{LY24} S.~Lando, Z.~Yang
{\it Chromatic polynomial and the $\so$-weight system},
arXiv:2411.01128


\bibitem{MM65}
J.~Milnor, J.~Moore,
{\it On the structure of Hopf algebras},
Ann. of Math. (2) 81 (1965),
211--264






\bibitem{OO96}
A.~Okounkov, G.~Olshanski,
\newblock {\it Shifted Schur Functions},
\newblock 1998,
\newblock St. Petersburg Math. J. volume 9:2, 239--300.

\bibitem{O91}
G.~Olshanski,
\newblock {\it Representations of
infinite-dimensional classical groups,
limits of enveloping algebras and yangians},
\newblock in “Topics in Representation Theory”,
Advances in Soviet Math. 2,
\newblock Amer. Math. Soc., Providence RI, 1991, pp. 1--66.

\bibitem{PP68}
A. M. Perelomov, V. S. Popov, 
\newblock{\it Casimir operators for semisimple
groups}, 
Math. USSR Izv. Volume 2(6), p.~1313, 1968




\bibitem{S83} M. Sato and Y. Sato,
\emph{Soliton equations as dynamical systems on infinite
dimensional Grassmann manifolds}, in: Nonlinear partial differential equations
in applied science (Tokyo 1982), North-HollandMath. Stud., vol. 81, North-
Holland, Amsterdam 1983, pp. 259--271.







\bibitem{V90}
 V.~Vassiliev,
{\it Cohomology of knot spaces}
in: Advance in Soviet Math.
v.~1,
1990,
23--69



\bibitem{ZY23}
Zhuoke Yang,
\newblock {\it On the Lie superalgebra
$\gl(m|n)$ weight system},
Journal of Geometry and Physics,
Volume {\bf187}, May 2023, 104808

\bibitem{ZY22}
Zhuoke Yang,
\newblock {\it New approaches to $\gl(N)$ weight system},
Izvestiya Mathematics, {\bf 87}:6,  150--166 (2023), 
arXiv:2202.12225 (2022)


\bibitem{Yo14}
K. Yordzhev,
\newblock{On the cardinality of a factor set in the symmetric group},
Asian-European Journal of Mathematics, Vol. 7, No. 2 (2014) 1450027, doi: 10.1142/S1793557114500272












\end{thebibliography}
\end{document}